\documentclass[11pt,letterpaper, reqno]{article}
\usepackage[utf8]{inputenc}
\usepackage[T1]{fontenc}
\usepackage{amsfonts}
\usepackage{enumerate}
\usepackage{amssymb}
\usepackage{enumerate}
\usepackage{graphicx}
\usepackage{mathtools}
\usepackage{amsmath, amsthm}
\usepackage{tikz}
\usepackage{color}
\usepackage[affil-it]{authblk}
\usepackage[sort,numbers]{natbib}
\usepackage{bbm}
\usepackage[letterpaper, total={6.5in, 8.5in}]{geometry}
\usepackage{amsmath}
\usepackage{amsfonts}
\usepackage{amssymb}
\usepackage{bbm}
\usepackage{mathrsfs}
\usepackage[utf8]{inputenc}
\usepackage[english]{babel}
\usepackage{hyperref}
\usepackage{relsize}
\usepackage{accents}
\usepackage{centernot}
\usepackage{subcaption}
\usepackage{url}

\usetikzlibrary{arrows,calc}

\makeatletter
  \@addtoreset{chapter}{part}
  \@addtoreset{@ppsaveapp}{part}
\makeatother

\long\def\/*#1*/{}

\usepackage{palatino}
\usepackage{eulervm}
\numberwithin{equation}{section}

\usepackage{hyperref}
\hypersetup{
  colorlinks,
  linkcolor=blue,
  citecolor=red,
  filecolor=black,
  urlcolor=black,
  pdftitle={},
  pdfauthor={},
  pdfcreator={},
  pdfsubject={},
  pdfkeywords={}
}
\numberwithin{equation}{section}
\usepackage{cleveref}

\newcommand{\Pro}[1]{\ensuremath{\mathbb{P}\left(#1\right)}}
\newcommand{\expt}[1]{\ensuremath{\mathbb{E}\left[#1\right]}}

\newcommand{\E}{\ensuremath{\mathbb{E}}}

\newcommand{\N}{\ensuremath{\mathbb{N}}}

\newcommand{\Dmb}{{\mathbb{D}}}
\newcommand{\Emb}{{\mathbb{E}}}
\newcommand{\Nmb}{{\mathbb{N}}}

\newcommand{\Rmb}{{\mathbb{R}}}
\newcommand{\Smb}{{\mathbb{S}}}

\newcommand{\dif}{\mathrm{d}}

\newcommand{\xbd}{{\boldsymbol{x}}}

\newcommand{\Nbar}{{\bar{\mathcal{N}}}}
\newcommand{\one}{{\boldsymbol{1}}}
\newcommand{\Lmc}{{\mathcal{L}}}
\newcommand{\bbar}{{\bar{b}}}

\newcommand{\Smc}{{\mathcal{S}}}

\newcommand{\Ebf}{{\mathbf{E}}}
\newcommand{\Smcbar}{{\bar{\Smc}}}
\newcommand{\Pbf}{{\mathbf{P}}}
\newcommand{\inn}{{\mathrm{in}}}
\newcommand{\out}{{\mathrm{out}}}

\newcommand{\qq}{\boldsymbol{q}}

\newtheorem{theorem}{Theorem}
\newtheorem*{claim*}{Claim}

\newtheorem{lemma}[theorem]{Lemma}

\newtheorem{remark}{Remark}

\newtheorem{condition}{Condition}

\let\plainqed\qedsymbol

\numberwithin{equation}{section}
\numberwithin{theorem}{section}

\newcommand{\QQ}{\mathbf{Q}}

\newcommand{\FF}{\mathbf{F}}
\newcommand{\cN}{\mathcal{N}}

\begin{document}

\title{Supermarket Model on Graphs}
\author[$1$]{Amarjit Budhiraja\footnote{budhiraj@email.unc.edu \hspace{2.5cm}$^\dagger$debankur\_mukherjee@brown.edu\hspace{2cm}$^\ddagger$ruoyu@umich.edu}}
\author[$2$]{Debankur Mukherjee$\dagger$}
\author[$3$]{Ruoyu Wu$\ddagger$}
\affil[$1$]{
University of North Carolina, Chapel Hill, USA}
\affil[$2$]{
Brown University, USA}
\affil[$3$]{
University of Michigan, USA}

\renewcommand\Authands{, }

\date{\today}

\maketitle 

\begin{abstract} 
We consider a variation of the supermarket model in which the servers can communicate with their neighbors and where the neighborhood relationships are described in terms of a suitable graph. Tasks with unit-exponential service time distributions arrive at each vertex as independent Poisson processes with rate $\lambda$, and each task is irrevocably assigned to the shortest queue among the one it first appears and its $d-1$ randomly selected neighbors. This model has been extensively studied when the underlying graph is a clique in which case it reduces to the well known {\em power-of-$d$} scheme. In particular, results of Mitzenmacher (1996) and Vvedenskaya {\em et al.}~(1996) show that as the size of the clique gets large, the occupancy process associated with the queue-lengths at the various servers converges to a deterministic limit described by an infinite system of ordinary differential equations (ODE). In this work, we consider settings where the underlying graph need not be a clique and is allowed to be suitably sparse. We show that if the minimum degree approaches infinity (however slowly) as the number of servers $N$ approaches infinity, and the ratio between the maximum degree and the minimum degree in each connected component approaches $1$ uniformly, the occupancy process converges to the same system of ODE as the classical supermarket model. In particular, the asymptotic behavior of the occupancy process is insensitive to the precise network topology. We also study the case where the graph sequence is random, with the $N$-th graph given as an Erd\H{o}s-R\'enyi random graph on $N$ vertices with average degree $c(N)$. Annealed convergence of the occupancy process to the same deterministic limit is established under the condition $c(N)\to\infty$, and under a stronger condition $c(N)/ \ln N\to\infty$, convergence (in probability) is shown for almost every realization of the random graph.
\end{abstract}

\section{Introduction}
\paragraph{Background and motivation.} 
In this paper we analyze a variation of the supermarket model in which the servers can communicate with their neighbors and where the neighborhood relationships are described in terms of a suitable graph.
Specifically, consider a graph $G_N$ on $N$ vertices, where the vertices represent single-server queues. 
Tasks with unit-exponential service time distributions arrive at each server as independent Poisson processes of rate $\lambda$, and each task is irrevocably assigned to the shortest queue among the one it first appears and its $d-1$ randomly selected neighbors.

The above model has been extensively investigated in the case where $G_N$ is a clique.
In that case, each task is assigned to the shortest queue among $d\geq 2$ queues selected randomly from the entire system, which is commonly referred to as the `power-of-$d$' or JSQ($d$) scheme.
Since the servers are exchangeable when the underlying graph is a clique,
the system is  quite tractable via classical mean-field techniques.
Results in Mitzenmacher~\cite{Mitzenmacher1996,Mitzenmacher01} and Vvedenskaya
{\em et al.}~\cite{VDK96} show that for any fixed value of $d$, as the size of the clique gets large, the occupancy process associated with the queue-lengths at the various servers converges to a deterministic limit described by an infinite system of ordinary differential equations (ODE).
Moreover, even sampling as few as $d = 2$
servers yields significant performance enhancements over purely
random assignment ($d = 1$) as $N \to \infty$.
Specifically,  when $\lambda<1$,
the probability that there are $i$ or more tasks at a given queue in steady state
is proportional to $\lambda^{\frac{d^i - 1}{d - 1}}$ as $N \to \infty$,
and thus exhibits super-exponential decay in $\lambda$ as opposed to exponential decay
for the random assignment policy.

However, in many service systems the `$d$ choices' might be geographically constrained~\cite{FG16, G15}, 
and when a task arrives at any specific server, it becomes difficult, if not impossible, to fetch instantaneous state information from an arbitrarily selected $d-1$ servers.
This might give rise to a constrained network architecture that can be captured in terms of a graph.
Moreover, executing a task commonly involves the use of some data,
and storing such data for all possible tasks on all servers will
typically require an excessive amount of storage capacity~\cite{XYL16, WZYTZ16}.
The above issues  motivate consideration of sparser graph
topologies where tasks that arrive at a specific server~$i$ can only  be forwarded to a subset of the servers ${\mathcal N}_i$ that possess the data required to process the tasks.
For the tasks that arrive at server~$i$, the $d-1$ random choices must come from ${\mathcal N}_i$.
The subset ${\mathcal N}_i$ containing the peers of server~$i$ can be thought of as neighbors in some graph $G_N$.
Although the above scenario corresponds to a setting with directed graphs and in our paper we consider the case of undirected graphs,
our results extend in a straightforward manner to the setting of  directed graphs, see Remarks~\ref{rem:directed} and~\ref{rem:directed-random} for detailed discussions.
While considering load balancing schemes with sparse topologies is desirable from applications perspectives, the corresponding mathematical formulation, that results in systems that in general will not be exchangeable or have simple Markovian state descriptors, puts us outside the range of
 classical mean-field techniques, leading to a fairly uncharted territory from methodological standpoint, as further discussed below.

\paragraph{Related work.}
The study of the JSQ($d$) scheme in the context of large-scale  queueing networks was initiated by Mitzenmacher \cite{Mitzenmacher1996,Mitzenmacher01}
and Vvedenskaya {\em et al.}~\cite{VDK96}.
Since then, this scheme along with its many variations have been studied extensively in \cite{BLP13,BLP12,AR17, EG16, MBLW16-3, Ying17, LM06, LN05, G05, MPS02, BudFri2} and many more.
In the context of load balancing problems on graphs, 
\cite{G15,T98} examines the performance
on certain fixed-degree graphs and in particular ring topologies.
Their results demonstrate that the flexibility to forward tasks
to a few neighbors, or even just one, with possibly shorter queues
significantly improves the performance in terms of the waiting time and tail distribution of the queue length. 
This is similar to the power-of-two effect in the setting of cliques, but the results in \cite{G15,T98} also establish that
the performance is sensitive to  the underlying graph topology, and that selecting from a fixed set of $d - 1$ neighbors typically
does not match the performance of re-sampling $d - 1$ alternate
servers for each incoming task from the entire population.
Recently, Mukherjee {\em et al.}~\cite{MBL17} study the join-the-shortest queue (JSQ) policy on graphs, where each task joins the shortest queue among the one it first appears and \emph{all} its neighbors, and establishes that asymptotically, the performance of the JSQ policy on a clique can be achieved by much sparser topologies, provided the graph is suitably random in Erd\H{o}s-R\'enyi sense. 
We will contrast the results of the current paper with those obtained in~\cite{MBL17}  in greater detail in Section~\ref{sec:main}, see Remark~\ref{rem:compare-20}.
Relevant from a high level, queueing system topologies with limited flexibility 
were examined using quite different
techniques by Tsitsiklis and Xu \cite{TX15,TX11} in a dynamic
scheduling framework (as opposed to the load balancing context).
We refer to~\cite{BBLM18} for a recent survey on scalable load balancing algorithms.

If tasks do not get served and never depart but simply accumulate,
then our model as described above amounts to a so-called
balls-and-bins problem on a graph.
Viewed from that angle, a close counterpart of our problem is
studied in Kenthapadi and Panigrahy~\cite{KP06}, where, in our
terminology, each arriving task is routed to the shortest
of $d \geq 2$ randomly selected neighboring queues. 
In this setup they show that if each vertex in the underlying graph has degree $\Theta(N^\varepsilon)$, where $\varepsilon$ is not
too small, the maximum number of balls in a bin scales as
$\log(\log(N))/\log(d) + O(1)$.
This scaling is the same as in the case when the underlying graph
is a clique~\cite{ABKU94}.
In a more recent paper by Peres, Talwar, and Weider~\cite{PTW15} the balls-and-bins problem has been analyzed in the context of a $(1+\beta)$-choice process, where each ball goes to a random bin with probability $1-\beta$ and to the lesser loaded of the two bins corresponding to the nodes of a random edge of the graph with probability $\beta$.
In particular, for this process they show that the difference between the maximum number of balls in a bin and the typical number of balls in the bins  is $O(\log(N)/\sigma)$, where $\sigma$ is the edge expansion property of the underlying graph.
We refer to~\cite{W17} for a recent survey on the balls-and-bins literature.
\paragraph{Main contributions.}
In most of the load balancing literature on systems of single-server queues mentioned above, the primary tool has 
been  a convenient occupancy measure representation for the collection of queue-length processes 
associated with the various servers.
Specifically, under the assumption of exponential service time distributions, the number of queues with queue length at least $i$ at time $t$ denoted by $Q_i(t)$, for $i=1,2,\ldots$ forms a Markov process.
This occupancy process $\QQ(\cdot) = (Q_1(\cdot), Q_2(\cdot), \ldots)$ is then analyzed  using classical mean-field techniques as the number of servers becomes large.
The fundamental challenge in the analysis of load balancing on arbitrary graph topologies is that one cannot reduce the study to that for the state occupancy process since it is no longer a Markov process. In general,
one needs to keep track
of the evolution of the number of tasks at each vertex along with the information on neighborhood relationships.
This is a significant obstacle in using tools from classical mean-field analysis for such systems.
Consequently, results for load balancing queuing systems
on general graphs have to date remained scarce.
To the best of our knowledge, this is the first work to study rigorously the limits of the JSQ($d$) occupancy process for non-trivial graph topologies (i.e., other than a clique). 

In~\cite{MBL17}, where the tasks are assigned to the shortest queue among \emph{all} the neighbors, the authors used a stochastic coupling to compare the occupancy process for an arbitrary graph topology with that for the clique, and establish that under suitable assumptions on the {\em well-connectedness} of the graph topology, the  occupancy processes and their diffusion scaled versions
have to the same weak limits as for the clique. 
Loosely speaking, for the first convergence, the well-connectedness requires that for any $\varepsilon>0$, the neighborhood of any collection of $\varepsilon N$ vertices contains $N-o(N)$ vertices.
This ensures that on any finite time interval, the fraction of tasks not assigned to servers with the `fluid-scaled minimum queue length' is arbitrarily small.
Thus for large $N$ the occupancy process becomes nearly indistinguishable from that in a clique.
The coupling in~\cite{MBL17} is particularly tailored for schemes where on any finite time interval, most of the arrivals are assigned to one of the fluid-scaled shortest queues.
For the setting considered in the current work where a fixed number of servers are probed at each arrival,
developing analogous coupling methods  appears to be challenging.
To see this, observe that when all neighbors are probed at arrivals, it is clear that the queue lengths will be better balanced (in the sense of stochastic majorization) for a clique than any other graph topology.
In contrast, for the JSQ($d$) scheme with fixed $d$, even this basic property, namely that the performance of the system will be `optimal' if the topology is a clique, is not clear.
In this paper, we  take a very different approach, and analyze the evolution of the queue-length process at an arbitrary tagged server as the system size becomes large.
The main ingredient is a careful analysis of local occupancy measures associated with neighborhood of each server and to argue that under suitable conditions their asymptotic behavior is the same for all servers.

Our first result establishes that under fairly mild conditions on the graph topology $G_N$ (diverging minimum degree and a degree regularity condition, see Condition~\ref{cond-reg} and also Remark \ref{rmk:weaker_condition}), for suitable initial occupancy measure, for any fixed $d\geq 2$, the global occupancy state process for the JSQ($d$) scheme on $G_N$ has the same weak limit as that on a clique, as the number of vertices $N$ becomes large (see Theorem \ref{th:deterministic}).
Also, we show that the propagation of chaos property holds for this system, in the sense that the  queue lengths at any finite collection of tagged servers are statistically asymptotically independent, and the 
 queue-length process for each server converges in distribution (in the path space) to the corresponding McKean-Vlasov process (see Theorem \ref{th:tagged}).
We note that the class of graphs for which the above results hold includes arbitrary $d(N)$-regular graphs, where $d(N)\to\infty$ as $N\to\infty$.
As an immediate consequence of these results, we obtain that the same asymptotic performance of
a JSQ($d$) scheme on  cliques can be achieved by a much sparser graph in which
the number of connections is reduced by almost a factor $N$.  Such a result provides a significant improvement on network connectivity requirements and gives important insights for sparse network design.

When the graph sequence $\{G_N\}_{N\geq 1}$ is random with $G_N$ given as an Erd\H{o}s-R\'enyi random graph (ERRG) with average degree $c(N)$, we establish that for any $c(N)$ that diverges to infinity with $N$, the annealed law of the occupancy process converges weakly to the same limit as in the case of a clique.
For convergence of the quenched law, we require a somewhat more stringent  growth  condition on the average degree.
Specifically, we show that if $c(N)/\log(N)\to\infty$ as $N\to\infty$, then for almost every realization of the random graph the  quenched law of the state occupancy process  converges  to the same  limit as for the case of a clique.
Thus the above results show that the asymptotic performance for cliques  can be achieved by much sparser topologies, even when the connections are random.

In the classical setting of weakly interacting particle systems one considers a collection of
 $N$ stochastic processes on a clique, given as the solution of $N$ coupled stochastic differential equations, where the evolution of any particle at a given time instant depends on its own state and the empirical measure of all particles at that moment (see \cite{Sznitman1991, Kolokoltsov2010,KurtzXiong1999} and references therein).
The asymptotic behavior of the associated state occupancy  measures have been  well studied, including the law of large numbers, propagation of chaos properties, central limit theorems, and large and moderate deviation principles.
However, there is much less work for systems on general graphs except for some recent results for  weakly interacting diffusions on Erd\H{o}s-R\'enyi random graphs. Annealed law of large numbers and central limit theorems for such systems have been established in \cite{BhamidiBudhirajaWu2016} and quenched law of large numbers has been shown in \cite{Delattre2016}.
However these works do not study queuing systems of the form considered here.

\paragraph{Organization of the paper.}
The rest of the paper is organized as follows.
In Section~\ref{sec:main} we present the main results of this paper along with some remarks and discussion -- Subsections~\ref{ssec:det} and~\ref{ssec:random} contains the results for sequence of deterministic and random graphs, respectively.
The proofs of  the results in Section \ref{sec:main} are presented in Section~\ref{sec:proofs}. 
Finally, we conclude with a discussion of future research directions in Section~\ref{sec:conclusion}.

\paragraph{Notation.}
Let $[N] \doteq \{1,\dotsc,N\}$ for $N \in \Nmb$.
For any graph $G_N = (V_N,E_N)$, where $V_N$ is a finite set of vertices and $E_N \subset V_N\times V_N$ is the set of edges,  and $i,j \in V_N$, let
  $\xi_{ij}^N = 1$ if $(i,j)\in E_N$ and $0$ otherwise.
In this work, throughout $V_N =[N]$, $G_N$ is undirected, namely $\xi_{ij}^N=\xi_{ji}^N$, and $E_N$ will be allowed to be random, in which case $\xi_{ij}^N$ will be random variables. 
Let $\Nmb_0 \doteq \Nmb \cup \{0\}$. For a set $A$, denote by $|A|$ the cardinality. 
For a Polish space $\Smb$, denote by $\Dmb([0,\infty),\Smb)$ the space of right continuous functions with left limits from $[0,\infty)$ to $\Smb$, endowed with the Skorokhod topology. 
For functions $f \colon [0,\infty) \to \Rmb$, let $\|f\|_{*,t} \doteq \sup_{0 \le s \le t} |f(s)|$.
We will use $\kappa,\kappa_1,\kappa_2,\dotsc$ for various non-negative finite constants.
The distribution of $\Smb$-valued random variable $X$ will be denoted as $\Lmc(X)$.
For $x \in \Smb$, denote by $\delta_x$ the Dirac measure at the point $x$.
When the underlying graph is non-random , expectations will be denoted by `$\E$', and when the graphs are random, the notation `$\Ebf$' will be used to denote the expectation (which integrates also over the randomness of the graph topology).

\section{Model description and main results}\label{sec:main}
Let $\{G_N = (V_N, E_N)\}_{N\geq 1}$ be a sequence of simple graphs where recall that $V_N = [N]$.
The graph $G_N$ corresponds to a system with $N$ servers, where each vertex in the graph  represents a server and edges in the graph define the neighborhood relationships.  
Tasks arrive at the various servers as independent Poisson processes of rate $\lambda$.
 Each server has its own queue with an infinite buffer. Fix $d \in \Nmb$, $d\ge 2$.
 When a task appears at a server $i$, it is immediately assigned to the server with the shortest queue among server $i$ and $d-1$ servers selected uniformly at random from its neighborhood in $G_N$.
If there are multiple such servers, one of them is chosen uniformly at random.
Arrivals to any server having less than $d-1$ neighbors in $G_N$ can be assigned in an arbitrary fashion among that server and its neighbors, e.g.
to itself (i.e., without probing the queue length at any other server).
The tasks have independent unit-mean exponentially distributed service times.
The service order at each of the queues is taken to be oblivious to the actual service time requirements.

Let $X_i^N(t)$ be the number of tasks at the $i$-th server at time instant $t$, starting from some a.s.~finite $X_i^N(0)$, and $q^N_j(t)$ be the fraction of servers with queue length at least $j$ in the $N$-th system at time $t$, $i\in [N]$, $j=1,2,\ldots$, namely
\begin{equation}
	q^N_j(t) \doteq \frac{1}{N} \sum_{i=1}^N \sum_{k=j}^{\infty} \one_{\{X^N_i(t)=k\}}, \; t \ge 0, \; j \in \Nmb_0.\label{eq:tailprob}
\end{equation}
Let $\qq^N(t) \doteq (q^N_i(t))_{i \in \mathbb{N}_0}$.
Then $\qq^N \doteq \{\qq^N(t)\}_{0\le t < \infty}$ is a process with sample paths in $\Dmb([0,\infty), S)$ where
$S = \{\qq\in [0,1]^\N: q_0=1, q_i\ge q_{i+1}\ \forall i\in \Nmb_0,\mbox{ and }\sum_{i}q_i<\infty\}$ is equipped with the $\ell_1$ topology.

We will now introduce a convenient representation for the evolution of the queue length processes in the $N$-th system. 
We begin by introducing some  notation. 
For $\xbd = (x_1,\dotsc,x_d) \in \Nmb_0^d$, let $b(\xbd)$ represent the probability that given $d$ servers chosen with queue lengths $\xbd$, the job is sent to the first server in the selection.
Recalling that the job is sent to the shortest queue with ties resolved by selecting at random, the precise definition is as follows:
\begin{equation}
	\label{eq:b}
	b(\xbd) \doteq \sum_{k=1}^d \frac{1}{k} \one_{\displaystyle \{x_1 = \min_{i \in [d]} \{x_i\}, |\mbox{argmin} \{x_i\}|=k \}}.
\end{equation}
Note that 
(i) $b(\xbd)$ is symmetric in $(x_2,\dotsc,x_d)$,
(ii) $b(\xbd) \in [0,1]$, and 
(iii) $b(\xbd)$ is $1$-Lipschitz in $\xbd \in \Nmb_0^d$.
Denote by $D^N_i$ the number of neighbors of a vertex $i$ in $G_N$.
Let $\cN_i$ be iid Poisson processes of rate 1, corresponding to service completions,
and $\Nbar_i$ be iid Poisson random measures on $[0,\infty) \times \Rmb_+$ with intensity $\lambda \, ds \, dy$.
Assume that $\{\cN_i, \Nbar_i\}$ are mutually independent.
Thanks to the Poisson splitting property, the evolution of $X_i^N(t)$ can be written as follows:
\begin{equation}
	\label{eq:X_i_n}
	X_i^N(t) = X_i^N(0) - \int_0^t \one_{\{X_i^N(s-)>0\}} \, \cN_i(ds) + \int_{[0,t] \times \Rmb_+} \one_{\{0 \le y \le C_i^N(s-)\}} \, \Nbar_i(ds\,dy),
\end{equation}
where 
\begin{equation}\label{eq:cn}
\begin{split}
	C_i^N(t) & = \one_{\{D_i^N<d-1\}} \bbar_i^N((X_k^N(t))_{k \in [N]},(\xi_{kl}^N)_{k,l \in [N]}) \\
	&  + \one_{\{D_i^N \ge d-1\}} \sum_{(j_2,\dotsc,j_d) \in \Smc_i^N} \alpha^N(i; j_2, j_3, \ldots, j_d) b(X_i^N(t),X_{j_2}^N(t),\dotsc,X_{j_d}^N(t)) \\
	&  + (d-1) \sum_{(j_2,\dotsc,j_d) \in \Smc_i^N} \one_{\{D_{j_2}^N \ge d-1\}} \alpha^N(j_2; i, j_3, \ldots, j_d) b(X_i^N(t),X_{j_2}^N(t),\dotsc,X_{j_d}^N(t)) \\
	&  + \sum_{j_2 \in [N], j_2 \ne i} \one_{\{D_{j_2}^N<d-1\}}\xi^N_{ij_2} \bbar_{ij_2}^N((X_k^N(t))_{k \in [N]},(\xi^N_{kl})_{k,l \in [N]}), \\
	\alpha^N(i; j_2, j_3, \ldots, j_d)&\doteq \frac{\xi^N_{ij_2}\xi^N_{ij_3}\dotsm\xi^N_{ij_d}}{D_i^N(D_i^N-1)\dotsm(D_i^N-d+2)}\\
	\Smc_i^N & \doteq \{ (j_2,\dotsc,j_d) \in [N]^{d-1} : (i,j_2,\dotsc,j_d) \mbox{ are distinct } \}.
\end{split}
\end{equation}
Here $\bbar_i^N$ and $\bbar_{ij}^N$ are measurable functions with 
\begin{equation}
	\label{eq:bbar}
	\bbar_i^N\Big((X_k^N(t))_{k \in [N]},(\xi_{kl}^N)_{k,l \in [N]}\Big), \bbar_{ij}^N\Big((X_k^N(t))_{k \in [N]},(\xi^N_{kl})_{k,l \in [N]}\Big) \in [0,D_i^N+1], 
\end{equation}
which define the rules of assigning tasks when $D_i^N<d-1$ or $D_j^N<d-1$, respectively.
Precise form of these functions will not be important in our analysis.
The second term in the expression for $C^N_i(t)$ gives the probability that a job arriving at server $i$ (with $D^N_i\ge d-1$)
is in fact assigned to server $i$ itself, which will happen if server $i$ is one of the queues with minimal queue length among the $d-1$ randomly selected neighbors and itself, and it is the winner of the tie between queues with minimal queue-lengths in the selection. The third term corresponds to the probability that a job arriving at some other server (say $j_2$, with $D^N_{j_2}\ge d-1$) is assigned to server $i$, which will happen if $i$ is a neighbor of $j_2$, server $i$ is among the random selection of $d-1$ neighbors of $j_2$, it is also among the queues with minimal queue-length in the selection, and it wins the tie-breaker among queues with minimal queue-length in the selection.
We note that although $C_i^N(t)$ takes a complicated form, it essentially depends on the local empirical queue-length and degree distributions of the neighbors, 
which allows the implementation of a 
mean-field type argument. This is one of the key advantages of the Poisson random measure representation in equation \eqref{eq:X_i_n}.


\subsection{Scaling limits for deterministic graph sequences}\label{ssec:det}
In this section we will consider arbitrary deterministic graph sequences, and establish a scaling limit  when the graphs satisfy a certain `regularity' condition  as formulated in Condition~\ref{cond-reg} below.
For any graph $G$, let $d_{\min}(G)$ and $d_{\max}(G)$ denote the minimum and maximum degree, respectively.
\begin{condition}[Regularity of degrees]\label{cond-reg}
The sequence $\{G_N\}_{N\geq 1}$  satisfies the following.
\begin{enumerate}[{\normalfont (i)}]
\item $d_{\min}(G_N)\to\infty$ as $N\to\infty$.
\item $\max_{i\in [N]} \left | \sum_{j\in [N], j\neq i} \frac{\xi_{ji}^N}{D_j^N} - 1\right| \to 0$ as $N\to\infty$. 
\end{enumerate}
\end{condition}

\begin{remark}\normalfont
	\label{rmk:weaker_condition}
	Condition \ref{cond-reg}(ii) holds if for example, $d_{\max}(G_N)/d_{\min}(G_N)\to 1$ as $N\to\infty$, since
	\begin{align*}
		\frac{d_{\min}(G_N)}{d_{\max}(G_N)} \le \frac{D_i^N}{d_{\max}(G_N)} 
		\le \sum_{j \in [N], j \ne i} \frac{\xi_{ji}^N}{D_j^N} \le \frac{D_i^N}{d_{\min}(G_N)} \le \frac{d_{\max}(G_N)}{d_{\min}(G_N)}
	\end{align*}
	for each $i \in [N]$.
	But Condition \ref{cond-reg}(ii) also allows $G_N$ to have degrees of very different orders in different components of the graph. 
For example, if $\{\mathcal{C}^N_k\}_{k\geq 1}$ denote the  connected components of $G_N$, then
Condition~\ref{cond-reg} (ii) is satisfied if
\[\sup_{k\geq 1} \left|\frac{d_{\min}(\mathcal{C}_k^N)}{d_{\max}(\mathcal{C}_k^N)}- 1\right| \to 0 \qquad\mbox{as}\qquad N\to\infty.\]
\end{remark}

Our first  result establishes under Condition~\ref{cond-reg}, the convergence of the occupancy state process $\qq^N$ to the same deterministic limit as for the classical JSQ($d$) policy (i.e. the case when $G_N$ is a clique), as $N\to\infty$.
\begin{theorem}[Convergence of global occupancy states]
\label{th:deterministic}
Assume that the sequence of graphs $\{G_N\}_{N\geq 1}$ satisfies Condition~\ref{cond-reg}, and $\{X_i^N(0):i\in [N]\}$ is iid with $\Pro{X_i^N(0) \geq j}=q^{\infty}_j$, $j=1,2,\ldots,$ for some $\qq^{\infty}\in S$. 
Then on any finite time interval, the occupancy state process $\qq^N(\cdot)$ converges weakly with respect to Skorohod $J_1$ topology to the deterministic limit $\qq(\cdot)$ given by the unique solution to the set of ODE:
	\begin{equation}\label{eq:occ-deterministic}
	\frac{dq_i(t)}{dt} = \lambda [(q_{i-1}(t))^d - (q_i(t))^d] - (q_{i}(t)-q_{i+1}(t)),\quad i = 1, 2,\ldots,
	\end{equation}	 
	and $\qq(0) = \qq^\infty$.
\end{theorem}
\begin{remark}\normalfont
We make the following observations.
	\begin{enumerate}[\normalfont (i)]
\item Unique solvability of the system of equations	\eqref{eq:occ-deterministic} is a consequence of Lipschitz continuity of the right side.
Specifically, define the function $\FF(\cdot) = (F_1(\cdot), F_2(\cdot), \ldots)$ on $S$ as 
\[F_i(\qq) = \lambda(q_{i-1}^d - q_i^d) - (q_i - q_{i+1}), \quad i= 1,2, \ldots,\] 
with $\qq\in S$ and $F_i(\qq)$ being the $i$-th component of $F(\qq)$. 
It is easily seen that $F$ is  Lipschitz on $S$ (equipped with the  $\ell_1$ distance).
Standard results then imply that the system of ODE  defined by $\dif\qq(t)/\dif t = \FF(\qq)$ admits a unique solution.

\item The above result shows in particular  that the evolution of the asymptotic global occupancy process as described by~\eqref{eq:occ-deterministic} coincides with that when the underlying graph is a clique, i.e., when each arriving task can probe any set of $d$ servers.
Thus under Condition~\ref{cond-reg}, the system exhibits the same asymptotic transient performance even when the underlying graph is much sparser.
As an immediate corollary we see that \eqref{eq:occ-deterministic} describes  the asymptotic system occupancy process associated with {\em arbitrary} $d(N)$-regular graphs as long as $d(N)\to\infty$ as $N\to\infty$.
\end{enumerate}
\end{remark}

\begin{remark}\normalfont\label{rem:compare-20}
Now we contrast Condition \ref{cond-reg} with
the condition introduced in \cite{MBL17} for the JSQ policy on a graph to behave as that on a clique.
We note that Condition~\ref{cond-reg} relies only on local properties of the graph, and in particular may hold even when, for example, the graph contains several connected components of sizes that grow to infinity with $N$.
In contrast, the condition in~\cite{MBL17} requires that any two $\Theta(N)$-sized component must share $\Theta(N)$ cross-edges, which does not hold in many networks with connectivity governed by spatial attributes, such as geometric graphs.
In this sense, Condition~\ref{cond-reg} includes much broader class of graphs including arbitrary $d(N)$-regular graphs with  $d(N)\to\infty$, as mentioned above.
On the other hand, our condition requires the minimum degree in the graph to diverge to infinity, whereas~\cite{MBL17} allows any $o(N)$ vertices to have bounded degree (or degree zero).
As noted in the introduction, it is easy to see that the queue length process of the JSQ policy on a clique is better balanced (in stochastic majorization sense) than on any other graph. 
This is also reflected by the fact that the sufficient criterion for fluid optimality as developed in~\cite{MBL17} is monotone with respect to edge addition.
Specifically, let $\{G_N=(V_N, E_N)\}_{N\geq 1}$ be a graph sequence which satisfies the sufficient criterion in~\cite{MBL17} for the limit of the  occupancy process coincides with that for cliques. 
Then \cite{MBL17} shows that for any graph sequence $\{\bar{G}_N = (V_N, \bar{E}_N)\}_{N\geq 1}$ with $E_N\subseteq\bar{E}_N$, the limit of the  occupancy process also coincides with that for cliques.
The above property is not immediate for systems considered in the current work since  adding edges arbitrarily may  result in violating Condition~\ref{cond-reg} (ii).
\end{remark}

\begin{remark}\label{rem:directed}\normalfont
Although in this paper we consider the case of undirected graphs, i.e., we assume $\xi_{ij}^N = \xi_{ji}^N$ for all $i,j\in [N]$, we note that the results naturally extends to the directed graph scenario.
Indeed, most of the proofs go through unchanged without the symmetry assumption for $\xi$ and remaining require minor modification.
In this remark we discuss how Condition~\ref{cond-reg} and the simplified condition in Remark~\ref{rmk:weaker_condition} need to be modified in the case of directed graphs in order for the conclusion of Theorem~\ref{th:deterministic} to hold.

For $i\in [N]$, denote $D_{i,\inn}^N\doteq \sum_{j\in[N]}\xi_{ji}^N$ and 
$D_{i,\out}^N\doteq \sum_{j\in[N]}\xi_{ij}^N$.
Also, for a graph $G_N$, denote $d_{\min}^{\inn}(G_N)\doteq \min_i D_{i,\inn}^N$ and $d_{\max}^{\inn}(G_N)\doteq \max_i D_{i,\inn}^N$, and define $d_{\min}^{\out}(G_N)$ and $d_{\max}^{\out}(G_N)$ similarly.
Assume that whenever a task arrives at a server~$i\in [N]$, it is immediately assigned to the server with the shortest queue among server $i$ and $d-1$ servers selected uniformly at random from the set of vertices $\{j:\xi_{ij}^N=1\}$ (ties are broken arbitrarily).
As before, arrivals to any server with $D_{i,\out}^N<d-1$ can be assigned arbitrarily.
Now, the conclusion of Theorem~\ref{th:deterministic} holds when the graph sequence $\{G_N\}_{N\geq 1}$ of directed graphs satisfies the following condition:  

\begin{quote}
\textbf{Condition 1`} (Criteria for directed graphs).
The graph sequence $\{G_N\}_{N\geq 1}$  satisfies the following.
\begin{enumerate}[{\normalfont (i)}]
\item $\min\big\{d_{\min}^{\inn}(G_N),d_{\min}^{\out}(G_N)\big\}\to\infty$ as $N\to\infty$.
\item $\max_{i\in [N]} \left | \sum_{j\in [N], j\neq i} \frac{\xi_{ji}^N}{D_{j,\out}^N} - 1\right| \to 0$ as $N\to\infty$. 
\end{enumerate}
\end{quote}
Similar to the condition discussed in Remark \ref{rmk:weaker_condition}, a weaker but simpler sufficient condition implying Condition 1` (ii) is the following: 
\[\frac{\min\big\{d_{\min}^{\inn}(G_N),d_{\min}^{\out}(G_N)\big\}}{\max\big\{d_{\max}^{\inn}(G_N),d_{\max}^{\out}(G_N)\big\}}\to 1\quad\mbox{as}\quad N\to\infty.\]
\end{remark}

Our second result gives the joint asymptotic  behavior of queue length processes for any finite collection of servers.
In particular, it shows that the propagation of chaos holds, i.e., the queue length processes for any finite collection of servers are asymptotically statistically independent.
Recall the sequence of Poisson processes $\{\cN_i\}$, Poisson random measures $\{\Nbar_i\}$, and the function $b$.

\begin{theorem}[Evolution of tagged servers]\label{th:tagged}
Assume that the  sequence of graphs $\{G_N\}_{N\geq 1}$ satisfies Condition~\ref{cond-reg},  and $\{X_i^N(0):i\in [N]\}$ is iid with $\Pro{X_i^N(0) \geq j}=q^{\infty}_j$, $j=1,2,\ldots,$ for some $\qq^{\infty}\in S$. 
Then the following convergence results hold.
\begin{enumerate}[{\normalfont (i)}]
\item On any finite time interval, the queue length process  $X_i^N(\cdot)$ at server $i$ converges weakly with respect to Skorohod $J_1$ topology to the following  McKean-Vlasov process:
\begin{equation}\label{eq:limit-tagged}
\begin{split}
	X_i(t) & = X_i(0) - \int_0^t \one_{\{X_i(s-)>0\}} \, \cN_i(ds) + \int_{[0,t]\times\Rmb_+} \one_{\{0 \le y \le C_i(s-)\}} \, \Nbar_i(ds\,dy), \\
	C_i(t) & = d\int_{\Nmb^{d-1}} b(X_i(t),x_2,\dotsc,x_d) \mu_t(dx_2)\dotsm\mu_t(dx_d) ,
\end{split}
\end{equation}
where $\mu_t = \Lmc(X_i(t))$ and $\mu_0[j,\infty) = q^{\infty}_j$ for $t\ge 0$ and $j \in \Nmb_0$.
\item For any  $m$-tuple $(i_1,\dotsc,i_m)\in \N^m$ with $i_j \ne i_k$ whenever $j \ne k$,
		\begin{equation*}
			\Lmc(X_{i_1}^N(\cdot),\dotsc,X_{i_m}^N(\cdot)) \to \mu^{\otimes m},
		\end{equation*}
		as probability measures on $\Dmb([0, \infty): \Nmb_0^m)$ where $\mu$ is the probability law of $X_1(\cdot)$ in part (i).
	\item For any $i\in\N$, the process $\mu^{i,N}$ denoting the occupancy measure process for the neighborhood of the $i$-th server, defined as
	\begin{equation}
	\label{eq:nbdempmzr}
	\mu^{i,N}_t \doteq   \frac{1}{D_i^N+1} \sum_{j \in [N], j \ne i} \xi_{ij}^N \delta_{X_j^N(t)} + 	\frac{1}{D_i^N+1}\delta_{X_i^N(t)}, \; t\ge 0,
	\end{equation}
	converges weakly with respect to Skorohod $J_1$ topology to the deterministic limit $\mu_{\cdot}$, where for $t\ge 0$,
	$\mu_t$ is as in part (i).
\end{enumerate}
\end{theorem}
\begin{remark}\label{rem:rem2}\normalfont
We note the following.
\begin{enumerate}[{\normalfont (i)}]
		\item The existence and uniqueness of solutions to \eqref{eq:limit-tagged} can be proved by standard arguments using the boundedness and Lipschitz property of the functions $b$ and $x\mapsto \one_{\{x>0\}}$ on $\Nmb_0$.
		\item Using the propagation of chaos property and the fact that $\{X_i(t):i\in [N]\}$ are iid, it follows
		that
		the limit of the global occupancy measure  at any time instant $t$ is in fact the law of $X_i(t)$ for any fixed $i$.
		Therefore,
		$$\mu_t[j,\infty)= \Pro{X_i(t)\geq j} = q_j(t),\;  j  \in \Nmb_0, i \in \Nmb \mbox{ and } t \ge 0.$$
	\end{enumerate}
\end{remark}

\subsection{Scaling limits for random graph sequences}\label{ssec:random}
Next we will consider the scenario when the underlying graph topology is random.
We consider asymptotics of both annealed and quenched laws of the occupancy process and the queue length process at any tagged server.
The following is our main condition in the study of the annealed law.
\begin{condition}[Diverging mean degree]\label{cond:errg1}
 $\{G_N\}_{N\geq 1}$ is a sequence of Erd\H{o}s-R\'enyi random graphs (ERRG) where any two vertices share an edge with probability $p_N$, and  $Np_N\to\infty$ as $N\to\infty$. 	$\{G_N\}_{N\geq 1}$ is independent of $\{X^N_j(0), \cN_i, \Nbar_i, j \in [N], N \in \Nmb, i \in \Nmb\}$.
\end{condition}


\begin{theorem}[Asymptotics of annealed law]\label{thm:npn_rate}
Assume that the  sequence of graphs $\{G_N\}_{N\geq 1}$ satisfies Condition~\ref{cond:errg1}, and $\{X_i^N(0):i\in [N]\}$ is iid with $\Pro{X_i^N(0) \geq j}=q^{\infty}_j$, $j=0, 1,2,\ldots,$ for some $\qq^{\infty}\in S$. 
Then the following  hold.
\begin{enumerate}[{\normalfont (i)}]
\item For any $T\in (0,\infty)$
		\begin{equation}
		\label{eq:npn_rate}
		\sup_{N \ge 1} \max_{i \in [N]} \sqrt{Np_N} \Ebf \|X_i^N - X_i\|_{*,T}^2 < \infty,
		\end{equation}
		where $X_i$ is as defined in \eqref{eq:limit-tagged}.
\item For any $m$-tuple $(i_1,\dotsc,i_m)\in \N^m$ with $i_j \ne i_k$ whenever $j \ne k$, 		
\begin{equation*}
			\Lmc(X_{i_1}^N(\cdot),\dotsc,X_{i_m}^N(\cdot)) \to \mu^{\otimes m},
		\end{equation*}
		as probability measures on $\Dmb([0, \infty): \Nmb_0^m)$ where $\mu$ is as in Theorem \ref{th:tagged}.
	\item For any $i\in\N$, the law of the neighborhood occupancy measure process  defined as in \eqref{eq:nbdempmzr}
	converges weakly in Skorohod $J_1$ topology to the deterministic limit $\mu_{\cdot}$. 
\end{enumerate}
\end{theorem}
\begin{remark}\normalfont
We make the following observations.
	\begin{enumerate}
\item	In contrast to standard convergence results for weakly interacting diffusions (see e.g.~\cite{Sznitman1991} or \cite{BhamidiBudhirajaWu2016}), the estimate in \eqref{eq:npn_rate} gives a rate of convergence of $\sqrt{Np_N}$ instead of $Np_N$. The reason for this can be seen from the proof which shows that  the bound for the  quantity $\Ebf \|X_i^N - X_i\|_{*,T}^2$ is controlled by $\Ebf |C_i^N(s)-C_i(s)|$ rather than $\Ebf |C_i^N(s)-C_i(s)|^2$, due to the form of indicator function in the evolution of $X_i^N$ (cf.~\eqref{eq:X_i_n}).
	\item Condition needed for   Theorem~\ref{thm:npn_rate} should be contrasted with that for Theorems \ref{th:deterministic}
	and \ref{th:tagged}.
	In particular, for the study of the  annealed law asymptotics we only need information on the average degree rather than on the  maximal and minimal degree of the graph.
	\item   All the limit theorems established in the current work have the feature that as long as there is interaction between `enough' particles the asymptotic behavior is same as that of a fully connected system. In settings where the interaction graph is very sparse, one expects  different types of asymptotic behavior. Consider for example one extreme case   when the graph is a collection of disjoint cliques of size $d$.
	 In this case the system decomposes into i.i.d.~copies of a JSQ system with $d$ servers and the limit behavior is very different.  For example, a propagation of chaos result of the form in Theorem \ref{th:tagged}(ii) is clearly false.
\end{enumerate}
\end{remark}

 We will now consider the asymptotic behavior of the quenched law of the occupancy process. For this we formulate a condition that is stronger than the one used in the study of the annealed asymptotics.
\begin{condition}[Condition for quenched limit]\label{cond:errg2}
 $\{G_N\}_{N\geq 1}$ is a sequence of Erd\H{o}s-R\'enyi random graphs, such that in $G_N$ any two vertices share an edge with probability $p_N$, and  $Np_N/\ln(N)\to\infty$ as $N\to\infty$. $\{G_N\}_{N\geq 1}$ is independent of $\{X^N_j(0), \cN_i, \Nbar_i, j \in [N], N \in \Nmb, i \in \Nmb\}$.
\end{condition}
The following theorem provides, under the above condition, the asymptotic behavior of the quenched law.
\begin{theorem}[Asymptotics of quenched law]\label{thm:npn_rate_quench}
Assume that the  sequence of graphs $\{G_N\}_{N\geq 1}$ satisfies Condition~\ref{cond:errg2}, and $\{X_i^N(0):i\in [N]\}$ is iid with $\Pro{X_i^N(0) \geq j}=q^\infty_j$, $j=0,1,2,\ldots,$ for some $\qq^\infty\in S$ for all $N$. 
Then the convergence results as stated in Theorems~\ref{th:deterministic} and~\ref{th:tagged} hold for almost every realization of the random graph sequence. 
\end{theorem}

\begin{remark}\label{rem:directed-random}
\normalfont
Consider the directed ERRG where any two distinct vertices $i,j$ have a directed edge from $i$ to $j$, independently of all other distinct pairs of vertices,  with probability $p_N$. Suppose conditions analogous 
to Conditions ~\ref{cond:errg1} and~\ref{cond:errg2} hold for this directed ERRG. Then, by a minor modification of the proofs, the conclusions of Theorems 
~\ref{thm:npn_rate} and~\ref{thm:npn_rate_quench} continue to hold. In fact some arguments get simpler since instead of the identity $\xi_{ij}^N = \xi_{ji}^N$ we have the independence of $\xi_{ij}^N$  and
$\xi_{ji}^N$. 
\end{remark}
\section{Proofs}\label{sec:proofs}

\subsection{Proofs for deterministic graph sequences}
An overview of the proof idea is as follows.
First note that the queue length process at any two vertices can be exactly coupled to evolve identically if the occupancy measure of the corresponding neighborhoods are indistinguishable.
The main step is to show that if the graph sequence satisfies  Condition~\ref{cond-reg}, then the local occupancy measure associated with the neighborhood of every server over any finite time interval converges to the same limit as for the global occupancy measure, which in turn is the same as that when the whole system uses the ordinary JSQ($d$) policy and the graph is a clique. 
This ensures that the rate of arrival (exogenous + forwarded from the neighboring vertices) to a typical server is (asymptotically) the same as that in the clique case. 
Thus, the law of the number of tasks at each server, and consequently the global occupancy measure, converge to the same limit.
For technical convenience we will provide the proof of Theorem~\ref{th:tagged} first, and then use that to establish Theorem~\ref{th:deterministic}.

We will define the limiting processes $(X_i(\cdot))_{i\geq 1}$ and the pre-limit processes $(X_i^N(\cdot))_{i\geq 1}$ on the same probability space by taking the same sequence of Poisson processes $\{\cN_i\}$ and Poisson random measures $\{\Nbar_i\}$ in both cases.
Also, take $X_i^N(0) = X_i(0)$ for all $i\in [N]$, $N\geq 1$.
Using Condition~\ref{cond-reg} we can find a $N_0 \in \Nmb$ such that for all $N\ge N_0$
\begin{equation}
	\label{eq:degree_large_N}
	d_{\min}(G_N) \ge d, \;  \sup_{i\in [N]} \left | \sum_{j\in [N], j\neq i} \frac{\xi_{ji}^N}{D_j} - 1\right| \le \frac{1}{2}, \sup_{i\in [N]} \sup_{t \in [0,T]} \left|C_i^N(t)\right| \le 2d.
\end{equation}
For the rest of this section we will assume that $N\ge N_0$ and therefore, in particular, the first and fourth terms in the definition of $C_i^N(s)$ are zero and the indicators in the second and third terms can be replaced by $1$.
We will frequently suppress $N$ in the notation $D^N_i$ and $\xi_{ij}^N$ and write them as $D_i$ and $\xi_{ij}$ respectively. We begin with the following lemma. Proof is given at the end of the subsection.
\begin{lemma}
	\label{lem:prep_quench}
	Let for $i \in [N]$ and $s\in [0,T]$
		\begin{align*}
			U_s  \doteq \E \left(\Big[  \sum_{(j_2,\dotsc,j_d) \in \Smc_i^N} \alpha^N(i; j_2, j_3, \ldots, j_d)    
			   \Big(b(X_i(s),X_{j_2}(s),\dotsc,X_{j_d}(s)) - \frac{C_i(s)}{d} \Big) \Big]^2 \right)
		\end{align*}
	and
		\begin{align*}	
		V_s  \doteq \E \left( \Big[ \sum_{(j_2,\dotsc,j_d) \in \Smc_i^N}  \alpha^N(j_2; i, j_3, \ldots, j_d)      \Big(b(X_i(s),X_{j_2}(s),\dotsc,X_{j_d}(s))-\frac{C_i(s)}{d}\Big) \Big]^2 \right).
		\end{align*}
Under the conditions of Theorem~\ref{th:deterministic}, there exists $K \in (0,\infty)$ such that for every $s \in [0,T]$ and  $i \in [N]$,
	\begin{align}
		U_s   \le \frac{K}{d_{\min}(G_N)} , \;\;
		V_s   \le \frac{K}{d_{\min}(G_N)} \left( \sum_{j=1,j \ne i}^N   \frac{\xi_{ji}}{D_j} \right)^2 . \label{eq:prep_quench_i}
	\end{align}
\end{lemma}

\begin{proof}[Proof of Theorem~\ref{th:tagged}]
Fix any $i\in\N$ and $T>0$.
From \eqref{eq:X_i_n} and \eqref{eq:limit-tagged}, using Cauchy--Schwarz and Doob's inequalities we have for any fixed $t\in [0,T]$ and $N\ge i$,	
\begin{align}
	\E \left\|X_i^N - X_i\right\|_{*,t}^2 & \le \kappa_1 \E \int_0^t |\one_{\{X_i^N(s)>0\}} - \one_{\{X_i(s)>0\}}|^2 \, ds + \kappa_1 \E \left( \int_0^t |\one_{\{X_i^N(s)>0\}} - \one_{\{X_i(s)>0\}}| \, ds \right)^2 \notag \\
	& \quad + \kappa_1 \E \int_{[0,t]\times\Rmb_+} |\one_{\{0 \le y \le C_i^N(s)\}} - \one_{\{0 \le y \le C_i(s)\}}|^2 ds\,dy \notag \\
	& \quad + \kappa_1 \E \left(\int_{[0,t]\times\Rmb_+} |\one_{\{0 \le y \le C_i^N(s)\}} - \one_{\{0 \le y \le C_i(s)\}}| \,ds\,dy\right)^2 \notag \\
	& \le \kappa_1 \int_0^t \E |X_i^N(s) - X_i(s)|^2 \, ds + \kappa_1 \E \left( \int_0^t |X_i^N(s) - X_i(s)| \, ds \right)^2 \notag \\
	& \quad + \kappa_1 \int_0^t \E |C_i^N(s)-C_i(s)| \, ds + \kappa_1 \E \left( \int_0^t |C_i^N(s)-C_i(s)| \, ds \right)^2 \notag \\
	& \le \kappa_2 \int_0^t \E |X_i^N(s) - X_i(s)|^2 \, ds + \kappa_2 \int_0^t \E |C_i^N(s)-C_i(s)| \, ds \label{eq:npn_quench_1}
\end{align}
for some  $\kappa_1, \kappa_2 \in (0,\infty)$, where the last line uses \eqref{eq:degree_large_N} and the fact that $0 \le \frac{C_i(s)}{d} \le 1$.

Now we analyze the difference $|C_i^N(s)-C_i(s)|$ in \eqref{eq:npn_quench_1}. 
Note that by adding and subtracting terms we have
	\begin{equation}
		|C_i^N(s)-C_i(s)| \le |C_i^N(s)-C_i^{N,1}(s)| + |C_i^{N,1}(s)-C_i^{N,2}(s)| + |C_i^{N,2}(s)-C_i(s)|, \label{eq:npn_2}
	\end{equation}
	where
	\begin{align*}
			C_i^{N,1}(s) & =  \sum_{(j_2,\dotsc,j_d) \in \Smc_i^N} \alpha^N(i; j_2, j_3, \ldots, j_d) b(X_i(s),X_{j_2}(s),\dotsc,X_{j_d}(s)) \\
			& \quad + (d-1) \sum_{(j_2,\dotsc,j_d) \in \Smc_i^N}  \alpha^N(j_2; i, j_3, \ldots, j_d) b(X_i(s),X_{j_2}(s),\dotsc,X_{j_d}(s)) 
		\end{align*}
		and
		\begin{align*}
			C_i^{N,2}(s) & =   \sum_{(j_2,\dotsc,j_d) \in \Smc_i^N} \alpha^N(i; j_2, j_3, \ldots, j_d) \frac{C_i(s)}{d} 
			 + (d-1) \sum_{(j_2,\dotsc,j_d) \in \Smc_i^N}  \alpha^N(j_2; i, j_3, \ldots, j_d) \frac{C_i(s)}{d}.
		\end{align*}
We now analyze each term in \eqref{eq:npn_2}.
In particular,  we will use the Lipschitz property of $b$ to handle the term $|C^N_i-C^{N,1}_i|$, and then use the iid property of $X_i$'s to handle the term $|C^{N,1}_i-C^{N.2}_i|$. 

First consider $|C_i^N(s)-C_i^{N,1}(s)|$.
From the Lipschitz property of $b$ and the definition of $\alpha^N$ we have
	\begin{align*}
		\E |C_i^N(s)-C_i^{N,1}(s)|
		& \le \E \Big[  \sum_{(j_2,\dotsc,j_d) \in \Smc_i^N} \alpha^N(i; j_2, j_3, \ldots, j_d)  \\ 
		& \qquad (|X_i^N(s)-X_i(s)| + |X_{j_2}^N(s)-X_{j_2}(s)| + \dotsb + |X_{j_d}^N(s)-X_{j_d}(s)|) \\
		& \qquad + (d-1) \sum_{(j_2,\dotsc,j_d) \in \Smc_i^N}  \alpha^N(j_2; i, j_3, \ldots, j_d)\\
		& \qquad  (|X_i^N(s)-X_i(s)| + |X_{j_2}^N(s)-X_{j_2}(s)| + \dotsb + |X_{j_d}^N(s)-X_{j_d}(s)|) \Big], \\
		& \le  \max_{j \in [N]} \E |X_j^N(s)-X_j(s)|\Big(d   + (d-1)d  \sum_{j_2 \in [N], j_2 \ne i}  \frac{\xi_{j_2i}}{D_{j_2}}\Big).
	\end{align*}
From \eqref{eq:degree_large_N} we have
	\begin{equation}
		\label{eq:npn_quench_difference_1}
		\E |C_i^N(s)-C_i^{N,1}(s)| \le \kappa_3 \max_{j \in [N]} \E |X_j^N(s)-X_j(s)|
	\end{equation}
for some  $\kappa_3 \in (0,\infty)$.
Next we consider $|C_i^{N,1}(s)-C_i^{N,2}(s)|$.
It follows from Cauchy--Schwarz inequality  that
	\begin{align*}
		&\E |C_i^{N,1}(s)-C_i^{N,2}(s)|^2\\
		& \quad\le 2 \E \Big[  \sum_{(j_2,\dotsc,j_d) \in \Smc_i^N} \alpha^N(i; j_2, j_3, \ldots, j_d)   
		 \Big(b(X_i(s),X_{j_2}(s),\dotsc,X_{j_d}(s)) - \frac{C_i(s)}{d}\Big) \Big]^2  \\
		& \quad + 2(d-1)^2 \E \Big[ \sum_{(j_2,\dotsc,j_d) \in \Smc_i^N}  \alpha^N(j_2; i, j_3, \ldots, j_d)  
		  \Big(b(X_i(s),X_{j_2}(s),\dotsc,X_{j_d}(s))-\frac{C_i(s)}{d}\Big) \Big]^2   \\
		 &\quad\le \kappa_4(U_s + V_s).
	\end{align*}
where $U_s,V_s$ are as in Lemma \ref{lem:prep_quench}.
From Lemma~\ref{lem:prep_quench} and \eqref{eq:degree_large_N} we obtain
\begin{align}
	\label{eq:npn_quench_difference_2}
\big(\E |C_i^{N,1}(s)-C_i^{N,2}(s)|\big)^2&\leq \E |C_i^{N,1}(s)-C_i^{N,2}(s)|^2 \\
&\le \frac{\kappa_5}{d_{\min}(G_N)}  + \frac{\kappa_5 }{d_{\min}(G_N)} \left( \sum_{j=1,j \ne i}^N   \frac{\xi_{ji}}{D_j} \right)^2
\le \frac{\kappa_6}{d_{\min}(G_N)}.
\end{align}
Finally we consider $|C_i^{N,2}(s)-C_i(s)|$.
Using the fact that $0 \le \frac{C_i(s)}{d} \le 1$, we have
	\begin{equation}
		\label{eq:npn_quench_difference_3}
		\E |C_i^{N,2}(s)-C_i(s)|
		\le \E \left[ \frac{(d-1)C_i(s)}{d} \left| \sum_{j \in [N], j \ne i} \frac{\xi_{ji}}{D_j} -1 \right| \right] \le (d-1) \left| \sum_{j \in [N], j \ne i} \frac{\xi_{ji}}{D_j} -1 \right|.
	\end{equation}	
	Since $|X_i^N(s)-X_i(s)|$ is non-negative integer-valued, we have $|X_i^N(s)-X_i(s)| \le |X_i^N(s)-X_i(s)|^2$.
	Combining this and \eqref{eq:npn_quench_1} -- \eqref{eq:npn_quench_difference_3} yields
	\begin{align*}
		&\max_{i \in [N]} \E \left\|X_i^N - X_i\right\|_{*,t}^2 \\
		&\quad \le \kappa_7 \int_0^t \max_{i \in [N]} \E \left\|X_i^N - X_i\right\|_{*,s}^2 \, ds + \kappa_7 \left(\frac{1}{(d_{\min}(G_N))^{1/2}} + \max_{i \in [N]} \left| \sum_{j \in [N], j \ne i} \frac{\xi_{ji}}{D_j} -1 \right| \right).
	\end{align*}
	From Gronwall's lemma and Condition \ref{cond-reg} we have
	\begin{equation}
		\label{eq:added}
		\lim_{N \to \infty} \max_{i \in [N]} \E \left\|X_i^N - X_i\right\|_{*,T}^2 = 0,
	\end{equation}
	which gives Theorem~\ref{th:tagged} (i).

Given part (i), the proof of propagation of chaos property as stated in Theorem~\ref{th:tagged} (ii) follows from standard arguments (cf. \cite{Sznitman1991}), and hence is omitted.
Also, having established the asymptotic result in Theorem~\ref{th:tagged} (i), the proof of convergence of local occupancy measures as stated in Theorem~\ref{th:tagged} (iii) can be established using similar arguments as in \cite[Corollary 3.3]{BhamidiBudhirajaWu2016}.
\end{proof}
We now complete the proof of Theorem \ref{th:deterministic}.

\begin{proof}[Proof of Theorem~\ref{th:deterministic}]
	From the asymptotic result in \eqref{eq:added} it follows (cf.\ \cite{Sznitman1991} and \cite[Corollary3.3(b)]{BhamidiBudhirajaWu2016}) that $\qq^N(\cdot)$ converges weakly with respect to Skorohod $J_1$ topology to the deterministic limit $\tilde \qq(\cdot)$ given by $\tilde q_j(t) = \mu_t[j,\infty) = \Pro{X_i(t)\ge j}$ for all $j \in \Nmb_0$ and $t\ge 0$.
	However we provide a proof here for completeness.
	Fix $T < \infty$ and consider random measures
		$\mu^N \doteq \frac{1}{N} \sum_{i=1}^N \delta_{X_i^N(\cdot)}$ and $\bar\mu^N \doteq \frac{1}{N} \sum_{i=1}^N \delta_{X_i(\cdot)}$,
	on $\Smb \doteq \Dmb([0,T]:\Nmb_0)$.
	where $X_i$ is introduced in \eqref{eq:limit-tagged}.
	Denote by $d_{BL}(\cdot,\cdot)$ the bounded Lipschitz metric:
	\begin{equation*}
		d_{BL}(\nu_1,\nu_2) \doteq \sup_{\|f\|_{BL} \le 1} \left|\int_\Smb f d\nu_1 - \int_\Smb f d\nu_2\right|,
	\end{equation*}
	where $\|f\|_{BL} \doteq \max\{ \|f\|_\infty, \sup_{x \ne y} \frac{|f(x)-f(y)|}{d(x,y)}\}$.
	It suffices to show that $d_{BL}(\mu^N,\bar{\mu}^N) \to 0$ and $\bar{\mu}^N \to \mu$ in probability as $N \to \infty$.
	Note that
	\begin{align*}
		\Emb d_{BL}(\mu^N,\bar{\mu}^N) & = \Emb \sup_{\|f\|_{BL} \le 1} \left|\int_\Smb f d\mu^N - \int_\Smb f d\bar{\mu}^N\right| = \Emb \sup_{\|f\|_{BL} \le 1} \left|\frac{1}{N} \sum_{i=1}^N \left( f(X_i^N)-f(X_i) \right) \right| \\
		& \le \Emb \frac{1}{N} \sum_{i=1}^N \sup_{\|f\|_{BL} \le 1} |f(X_i^N)-f(X_i)| \le \frac{1}{N} \sum_{i=1}^N \Emb \left\|X_i^N - X_i\right\|_{*,T} \to 0
	\end{align*}
	as $N \to \infty$ by \eqref{eq:added}, and hence $d_{BL}(\mu^N,\bar{\mu}^N) \to 0$ in probability.
	Also note that for every bounded and continuous function $f$ on $\Smb$, from independence of $\{X_i\}$ we have
	\begin{align*}
		\Emb \left(\int_\Smb f d\bar{\mu}^N - \int_\Smb f d\mu\right)^2 & = \frac{1}{N^2} \sum_{i,j=1}^N Cov(f(X_i),f(X_j)) = \frac{1}{N^2} \sum_{i=1}^N Var(f(X_i)) \le \frac{\|f\|_\infty^2}{N} \to 0
	\end{align*}
	as $N \to \infty$, and hence $\bar{\mu}^N \to \mu$ in probability.
	Therefore $\mu^N \to \mu$ in probability as $N \to \infty$.
	One can easily check that $\sup_{N \ge 1} \Emb \sup_{0 \le t \le T} \|\qq^N(t)\|_{\ell_1}^2 < \infty$, which implies $\qq^N \to \tilde \qq$ as $N \to \infty$.

	Next, in order to prove the theorem it suffices to show that $\tilde \qq$ satisfies the system of ODE in \eqref{eq:occ-deterministic}.	
Define $f_j(x) = \one_{\{x\geq j\}}$, $j=1,2,\ldots$.
Then Equation~\eqref{eq:limit-tagged} yields
\begin{align*}
\E{f_j(X_i(t))} &= \E{f_j(X_i(0))}+\int_0^t \expt{\one_{\{X_i(s)>0\}}(f_j(X_i(s)-1)-f_j(X_i(s)))}\dif s\\
&\hspace{4cm}+ \lambda d \int_0^t \int_{\N^{d-1}}\E\Big[ b(X_i(s),x_2,\ldots,x_d)(f_j(X_i(s)+1)\\
&\hspace{6cm}-f_j(X_i(s)))\Big]\mu_s(\dif x_2)\ldots \mu_s(\dif x_d)\dif s\\
&=  \E{f_j(X_i(0))}-\int_0^t \expt{f_j(X_i(s))-f_{j+1}(X_i(s))}\dif s\\
&\hspace{4cm}+ \lambda d \int_0^t \int_{\N^{d-1}}\E \Big[b(j-1,x_2,\ldots,x_d)(f_{j-1}(X_i(s))\\
&\hspace{6cm}-f_j(X_i(s)))\Big]\mu_s(\dif x_2)\ldots \mu_s(\dif x_d)\dif s.
\end{align*}
Since $\E[f_j(X_i(t))]=\tilde q_j(t)$ for $j=1,2,\ldots$, we obtain
\begin{align}
\tilde q_j(t) &= \tilde q_j(0) - \int_0^t (\tilde q_j(s) - \tilde q_{j+1}(s))\dif s + \lambda d\int_0^t (\tilde q_{j-1}(s) - \tilde q_j(s)) \notag \\
&\hspace{5cm}\times\int_{\N^{d-1}} b(j-1,x_2,\ldots,x_d)\mu_s(\dif x_2)\ldots \mu_s(\dif x_d)\dif s. \label{eq:qtil_ODE}
\end{align}
Using \eqref{eq:b} and the fact that 
$\tilde q_j(t) = \mu_t[j,\infty) = \Pro{X_i(t)\ge j}$, $j=1,2,\ldots$, we have
\begin{align*}
	& d(\tilde q_{j-1}(s) - \tilde q_j(s)) \int_{\N^{d-1}} b(j-1,x_2,\ldots,x_d)\mu_s(\dif x_2)\ldots \mu_s(\dif x_d) \\
	& = d(\tilde q_{j-1}(s) - \tilde q_j(s)) \sum_{k=1}^d \frac{1}{k} \binom{d-1}{k-1} (\tilde q_{j-1}(s) - \tilde q_j(s))^{k-1} (\tilde q_j(s))^{d-k} \\
	& = \sum_{k=0}^d \binom{d}{k} (\tilde q_{j-1}(s) - \tilde q_j(s))^k (\tilde q_j(s))^{d-k} - (\tilde q_j(s))^d \\
	& = (\tilde q_{j-1}(s))^d - (\tilde q_j(s))^d.
\end{align*}
Therefore \eqref{eq:qtil_ODE} can be written as
\begin{equation*}
	\tilde q_j(t) = \tilde q_j(0) - \int_0^t (\tilde q_j(s) - \tilde q_{j+1}(s))\dif s + \lambda \int_0^t [(\tilde q_{j-1}(s))^d - (\tilde q_j(s))^d]\dif s.
\end{equation*}
This shows that $\tilde \qq$ satisfies the system of ODE in \eqref{eq:occ-deterministic} and completes the proof of Theorem~\ref{th:deterministic}.
\end{proof}

\begin{proof}[Proof of Lemma~\ref{lem:prep_quench}]
	We first show the first inequality in \eqref{eq:prep_quench_i}.
	Observe that
	\begin{align}
		U_s & = \sum_{(j_2,\dotsc,j_d) \in \Smc_i^N} \sum_{(k_2,\dotsc,k_d) \in \Smc_i^N} \left[  \alpha^N(i; j_2, j_3, \ldots, j_d)   \alpha^N(i; k_2, k_3, \ldots, k_d)  \right] \notag \\
		& \quad \E \left[ \left(b(X_i(s),X_{j_2}(s),\dotsc,X_{j_d}(s)) - \frac{C_i(s)}{d} \right) \left(b(X_i(s),X_{k_2}(s),\dotsc,X_{k_d}(s)) - \frac{C_i(s)}{d} \right) \right]. \notag
	\end{align}
Now observe that since $\{X_i(0):i\in [N]\}$ are iid, we have $\{X_i(s):i\in [N]\}$ are also iid for any fixed $s>0$.
Thus,
	\begin{equation}
		\label{eq:prep_quench_simplify}
		\E \left[ \left(b(X_i(s),X_{j_2}(s),\dotsc,X_{j_d}(s)) - \frac{C_i(s)}{d} \right) \left(b(X_i(s),X_{k_2}(s),\dotsc,X_{k_d}(s)) - \frac{C_i(s)}{d} \right) \right]=0
	\end{equation}
	when $(i,j_2,k_2,\dotsc,j_d,k_d)$ are distinct.
Therefore, we have
	\begin{align}
		U_s \le \sum  \alpha^N(i; j_2, j_3, \ldots, j_d) \alpha^N(i; k_2, k_3, \ldots, k_d), \label{eq:prep_quench_2}
	\end{align}
	where the summation is taken over
	\begin{equation}
		\label{eq:prep_quench_3}
		\hat \Smc_i^N \doteq \left\{(j_2,\dotsc,j_d) \in \Smc_i^N, (k_2,\dotsc,k_d) \in \Smc_i^N, (j_2,k_2,\dotsc,j_d,k_d) \mbox{ are not distinct}\right\}
	\end{equation}
	and the inequality follows since $0 \le b \le 1$ and $0 \le \frac{C_i(s)}{d} \le 1$.
	Since the total number of combinations in \eqref{eq:prep_quench_3} such that $(\xi_{ij_2}\xi_{ij_3}\dotsm\xi_{ij_d})(\xi_{ik_2}\xi_{ik_3}\dotsm\xi_{ik_d})=1$ is no more than
	\begin{equation}\label{eq:bdonbindiff}
		\left[(d-1)! \binom{D_i}{d-1}\right]^2 - (2d-2)!\binom{D_i}{2d-2} \le \kappa_1 D_i^{2d-3}, 
	\end{equation}
	we can bound \eqref{eq:prep_quench_2} by
	\begin{align*}
	\frac{\kappa_1 D_i^{2d-3}}{D_i^2(D_i-1)^2\dotsm(D_i-d+2)^2} \le \kappa_2  \frac{1}{D_i} \le \frac{\kappa_2}{d_{\min}(G_N)}.
	\end{align*}
	This gives the first bound in \eqref{eq:prep_quench_i}.
	
	Next we show the second bound in \eqref{eq:prep_quench_i}.
	From  \eqref{eq:prep_quench_simplify} it follows from the same argument used for \eqref{eq:prep_quench_2} that 
	\begin{equation}
		V_s \le \sum  \alpha^N(j_2; i, j_3, \ldots, j_d) \alpha^N(k_2; i, k_3, \ldots, k_d), \label{eq:prep_quench_4}
	\end{equation}
	where the summation is taken over \eqref{eq:prep_quench_3}.
	Since for fixed $(j_2,k_2) \in \Smcbar_i$, where
	\begin{equation}\label{eq:smcbar}
	\Smcbar_i\doteq \{ (j,k) \in [N]^2 : j \ne i, k \ne i\},
	\end{equation}
	the total number of combinations in \eqref{eq:prep_quench_3} such that $(\xi_{j_2i}\xi_{j_2j_3}\dotsm\xi_{j_2j_d})(\xi_{k_2i}\xi_{k_2k_3}\dotsm\xi_{k_2k_d})=1$ is no more than
	\begin{align}
		& \left[(d-2)! \binom{D_{j_2}-1}{d-2}\right] \left[ (d-2)! \binom{D_{k_2}-1}{d-2} \right] - \left[(d-2)!\binom{D_{j_2}-2}{d-2}\right] \left[(d-2)!\binom{D_{k_2}-d}{d-2}\right] \nonumber\\
		& \le \kappa_3 (D_{j_2}^{d-3}D_{k_2}^{d-2} + D_{j_2}^{d-2}D_{k_2}^{d-3}), \label{eq:d2d3bd}
	\end{align}
	where the second term in the first line corresponds to choosing distinct $j_3,\dotsc,j_d$ from $D_{j_2}-2$ neighbors (excluding $i,k_2$) of $j_2$ and then choosing distinct $k_3,\dotsc,k_d$ from $D_{k_2}-d$ neighbors (excluding $i,j_2,\dotsc,j_d$) of $k_2$.
	Now, we can bound \eqref{eq:prep_quench_4} by
	\begin{align*}
		 &\sum_{(j_2,k_2) \in \Smcbar_i}  \frac{\kappa_3 (D_{j_2}^{d-3}D_{k_2}^{d-2} + D_{j_2}^{d-2}D_{k_2}^{d-3})\xi_{j_2i}\xi_{k_2i}}{D_{j_2}(D_{j_2}-1)\dotsm(D_{j_2}-d+2)D_{k_2}(D_{k_2}-1)\dotsm(D_{k_2}-d+2)}\\
		& \le \kappa_4 \sum_{(j_2,k_2) \in \Smcbar_i}  \left( \frac{\xi_{j_2i}\xi_{k_2i}}{D_{j_2}^2D_{k_2}} + \frac{\xi_{j_2i}\xi_{k_2i}}{D_{j_2}D_{k_2}^2} \right) \\
		& \le \kappa_4 \frac{2}{d_{\min}(G_N)} \left( \sum_{j=1,j \ne i}^N   \frac{\xi_{ji}}{D_j} \right)^2.
	\end{align*}
	This completes the proof.
\end{proof}

\subsection{Proofs for random graph sequences}
In this section we give the proofs of Theorems \ref{thm:npn_rate} and \ref{thm:npn_rate_quench}.
As in the proof of Theorem~\ref{th:tagged}, we will define the limiting processes $(X_i(\cdot))_{i\geq 1}$ and the pre-limit processes $(X_i^N(\cdot))_{i\geq 1}$ on the same probability space by taking identical sequence of Poisson processes $\{N_i\}$ and Poisson random measures $\{\Nbar_i\}$ in both cases. The random graph sequence $\{G_N\}$ will also be given on this common probability space and is taken to be independent of the Poisson processes and Poisson random measures.
Finally, we take $X_i^N(0) = X_i(0)$ for all $i\in [N]$, $N\geq 1$.
Once again, we will frequently suppress $N$ in the notation $D^N_i$ and write it as $D_i$.
We begin with three lemmas that will be used in the proof.
Let for $s\ge 0$
	\begin{align}
		U_s^A & \doteq \Ebf\left( \Big[ \one_{\{D^N_i \ge d-1\}} \sum_{(j_2,\dotsc,j_d) \in \Smc_i^N} \alpha^N(i; j_2, j_3, \ldots, j_d)   \Big(b(X_i(s),X_{j_2}(s),\dotsc,X_{j_d}(s)) - \frac{C_i(s)}{d} \Big) \Big]^2 \right)
		\label{eq:prep_new_i} 
	\end{align}
and
	\begin{align}
		V_s^A & \doteq \Ebf\left( \Big[ \sum_{(j_2,\dotsc,j_d) \in \Smc_i^N} \one_{\{D^N_{j_2} \ge d-1\}} \alpha^N(j_2; i, j_3, \ldots, j_d) \Big(b(X_i(s),X_{j_2}(s),\dotsc,X_{j_d}(s))-\frac{C_i(s)}{d}\Big) \Big]^2\right). 
		\label{eq:prep_new_j}
	\end{align}	
	Note that the dependence of $U_s^A$ and $V_s^A$ on $i$ is suppressed in the notation.
The next lemma provides uniform bounds on $U_s^A$ and $V_s^A$.
\begin{lemma}
	\label{lem:prep_new}
	Fix $T\ge 0$. Under the conditions of Theorem~\ref{thm:npn_rate}, there exists $\kappa \in (0,\infty)$ such that for every $s \in [0,T]$ and  $i \in [N]$,
	\begin{align}
		U_s^A  \le \frac{\kappa}{Np_N} \qquad\mbox{and}\qquad 
		V_s^A \le \frac{\kappa}{Np_N} + \frac{\kappa}{(Np_N)^2}.\notag 
	\end{align}
\end{lemma}
\noindent
The proof of Lemma~\ref{lem:prep_new} follows along similar lines  as the proof of Lemma~\ref{lem:prep_quench}, however note that the expectations in \eqref{eq:prep_new_i} and \eqref{eq:prep_new_j} are taken also over the randomness of the graph topology, and thus we need additional arguments. 
Proof of Lemma~\ref{lem:prep_new} is provided at the end of this subsection.

The next lemma is taken from \cite{BhamidiBudhirajaWu2016}.
\begin{lemma}[{\cite[Lemma~5.2]{BhamidiBudhirajaWu2016}}]
	\label{lem:prep_2}
	Let $G_N$ be an ERRG with connection probability $p_N$. Then
	\begin{equation*}
		\Ebf\ \Big( \sum_{j \in [N], j \ne i} \frac{\xi_{ij}^N}{D_j^N} \one_{\{D_j^N > 0\}} - 1 \Big)^2 \le \frac{4}{N p_N} + 2 e^{-N p_N}, \quad i \in [N],
	\end{equation*}
\end{lemma}

The following lemma provides useful moment bounds on $|X_i^N - X_i|$ and its proof is given at the end of this subsection.
\begin{lemma}
	\label{lem:momentbd}
	Fix $T\ge 0$. Under the conditions of Theorem~\ref{thm:npn_rate},
	\begin{equation*} 
		\sup_{N \ge 1} \max_{i \in [N]} \Ebf \left\|X_i^N - X_i\right\|_{*,T}^4 < \infty.
	\end{equation*}
\end{lemma}

We now present the proof of Theorem \ref{thm:npn_rate}.

\begin{proof}[Proof of Theorem \ref{thm:npn_rate}]

Fix any $i\in\N$ and $T>0$.
From \eqref{eq:X_i_n} and \eqref{eq:limit-tagged}, using Cauchy--Schwarz and Doob's inequalities we have for any fixed $t\in [0,T]$	
\begin{align}
	\Ebf \left\|X_i^N - X_i\right\|_{*,t}^2 & \le 
	\kappa_1 \int_0^t \Ebf |X_i^N(s) - X_i(s)|^2 \, ds + \kappa_1 \int_0^t \Ebf |C_i^N(s)-C_i(s)| \, ds + \kappa_1 \int_0^t \Ebf |C_i^N(s)-C_i(s)|^2 \, ds \label{eq:npn_1}
\end{align}
for some  $\kappa_1\in (0,\infty)$.
Define $C^{N,1}_i(s)$ and $C^{N,2}_i(s)$ by
\begin{align*}
	C_i^{N,1}(s) & = \one_{\{D_i<d-1\}} \bbar_i((X_k^N(s))_{k \in [N]},(\xi_{kl})_{k,l \in [N]}) \\
	& \quad + \one_{\{D_i \ge d-1\}} \sum_{(j_2,\dotsc,j_d) \in \Smc_i^N} \alpha^N(i; j_2, j_3, \ldots, j_d) b(X_i(s),X_{j_2}(s),\dotsc,X_{j_d}(s)) \\
	& \quad + (d-1) \sum_{(j_2,\dotsc,j_d) \in \Smc_i^N} \one_{\{D_{j_2} \ge d-1\}} \alpha^N(j_2; i, j_3, \ldots, j_d) b(X_i(s),X_{j_2}(s),\dotsc,X_{j_d}(s)) \\
	& \quad + \sum_{j_2 \in [N], j_2 \ne i} \one_{\{D_{j_2<d-1}\}}\xi_{ij_2} \bbar_{ij_2}((X_k^N(s))_{k \in [N]},(\xi_{kl})_{k,l \in [N]})
\end{align*}
and
\begin{align*}
	C_i^{N,2}(s) & = \one_{\{D_i<d-1\}} \bbar_i((X_k^N(s))_{k \in [N]},(\xi_{kl})_{k,l \in [N]}) \\
	& \quad + \one_{\{D_i \ge d-1\}} \sum_{(j_2,\dotsc,j_d) \in \Smc_i^N} \alpha^N(i; j_2, j_3, \ldots, j_d) \frac{C_i(s)}{d} \\
	& \quad + (d-1) \sum_{(j_2,\dotsc,j_d) \in \Smc_i^N} \one_{\{D_{j_2} \ge d-1\}} \alpha^N(j_2; i, j_3, \ldots, j_d) \frac{C_i(s)}{d} \\
	& \quad + \sum_{j_2 \in [N], j_2 \ne i} \one_{\{D_{j_2<d-1}\}}\xi_{ij_2} \bbar_{ij_2}((X_k^N(s))_{k \in [N]},(\xi_{kl})_{k,l \in [N]}).
\end{align*}
By adding and subtracting terms we have ~\eqref{eq:npn_2} and
\begin{equation}
	|C_i^N(s)-C_i(s)|^2 \le 3|C_i^N(s)-C_i^{N,1}(s)|^2 + 3|C_i^{N,1}(s)-C_i^{N,2}(s)|^2 + 3|C_i^{N,2}(s)-C_i(s)|^2. \label{eq:npn_3}
\end{equation}
Here although one has $\Ebf |C_i^N(s)-C_i(s)| \le \left( \Ebf |C_i^N(s)-C_i(s)|^2 \right)^{1/2}$, in order to get the desired rate $\sqrt{Np_N}$ in \eqref{eq:npn_rate}, we have to estimate $\Ebf |C_i^N(s)-C_i(s)|$ more carefully through \eqref{eq:npn_2}.

Let us consider $|C_i^N(s)-C_i^{N,1}(s)|$ and $|C_i^N(s)-C_i^{N,1}(s)|^2$ first.
We claim that for $m=1,2$, there exists some $\kappa_2 \in (0,\infty)$ such that
\begin{align}
	\Ebf |C_i^N(s)-C_i^{N,1}(s)|^m & \le \kappa_2 \Ebf |X_i^N(s)-X_i(s)|^m + \kappa_2 \Ebf \Big[ \one_{\{D_i \ge d-1\}} \sum_{j \in [N], j \ne i} \frac{\xi_{ij}}{D_i} |X_j^N(s)-X_j(s)|^m \Big] \notag \\
	& \quad + \kappa_2 \left(\frac{1}{Np_N} + e^{-Np_N}\right)^{1/2}. \label{eq:CN_CN1}
\end{align}
To see this, note that from the Lipschitz property of $b$ and the definition of $\Smc_i^N$ we have
	\begin{align*}
		\Ebf |C_i^N(s)-C_i^{N,1}(s)| & \le \Ebf \Big[ \one_{\{D_i \ge d-1\}} \sum_{(j_2,\dotsc,j_d) \in \Smc_i^N} \alpha^N(i; j_2, j_3, \ldots, j_d) \notag \\ 
		& \qquad (|X_i^N(s)-X_i(s)| + |X_{j_2}^N(s)-X_{j_2}(s)| + \dotsb + |X_{j_d}^N(s)-X_{j_d}(s)|) \notag \\
		& \qquad + (d-1) \sum_{(j_2,\dotsc,j_d) \in \Smc_i^N} \one_{\{D_{j_2} \ge d-1\}} \alpha^N(j_2; i, j_3, \ldots, j_d) \notag \\
		& \qquad  (|X_i^N(s)-X_i(s)| + |X_{j_2}^N(s)-X_{j_2}(s)| + \dotsb + |X_{j_d}^N(s)-X_{j_d}(s)|) \Big], \notag \\
		& = d\ \Ebf \Big[ \one_{\{D_i \ge d-1\}} \sum_{(j_2,\dotsc,j_d) \in \Smc_i^N} \alpha^N(i; j_2, j_3, \ldots, j_d)   \notag \\ 
		& \qquad  (|X_i^N(s)-X_i(s)| + |X_{j_2}^N(s)-X_{j_2}(s)| + \dotsb + |X_{j_d}^N(s)-X_{j_d}(s)|) \Big] \notag \\
		& \le d\ \Ebf |X_i^N(s)-X_i(s)| + d(d-1) \Ebf \Big[ \one_{\{D_i \ge d-1\}} \sum_{j \in [N], j \ne i} \frac{\xi_{ij}}{D_i} |X_j^N(s)-X_j(s)| \Big], 
	\end{align*}
	where in obtaining the equality we have used the exchangeability property:
	\begin{align}
		& \Lmc(\xi_{ij_2},\xi_{ij_3},\dotsc,\xi_{ij_d},D_i, X_i^N(s),X_i(s),X_{j_2}^N(s),X_{j_2}(s),X_{j_3}^N(s),X_{j_3}(s),\dotsc,X_{j_d}^N(s),X_{j_d}(s)) \notag \\
		& = \Lmc(\xi_{j_2i},\xi_{j_2j_3},\dotsc,\xi_{j_2j_d},D_{j_2}, X_{j_2}^N(s),X_{j_2}(s),X_i^N(s),X_i(s),X_{j_3}^N(s),X_{j_3}(s),\dotsc,X_{j_d}^N(s),X_{j_d}(s)) \label{eq:exchangeability}
	\end{align}	
	for $(j_2,\dotsc,j_d) \in \Smc_i^N$.
	Therefore the claim \eqref{eq:CN_CN1} holds for $m=1$.
	Next we verify \eqref{eq:CN_CN1} when $m=2$.
	Note that
	\begin{align*}
		\Ebf |C_i^N(s)-C_i^{N,1}(s)|^2 & \le 2R_i^{N,1}(s) + 2(d-1)^2R_i^{N,2}(s),
	\end{align*}
	where
	\begin{align*}
		R_i^{N,1}(s) & \doteq \Ebf \Big[ \one_{\{D_i \ge d-1\}} \sum_{(j_2,\dotsc,j_d) \in \Smc_i^N} \alpha^N(i; j_2, j_3, \ldots, j_d) \\
		& \qquad  [b(X_i^N(s),X_{j_2}^N(s),\dotsc,X_{j_d}^N(s)) - b(X_i(s),X_{j_2}(s),\dotsc,X_{j_d}(s))] \Big]^2, \notag \\
		R_i^{N,2}(s) & \doteq \Ebf \Big[ \sum_{(j_2,\dotsc,j_d) \in \Smc_i^N} \one_{\{D_{j_2} \ge d-1\}} \alpha^N(j_2; i, j_3, \ldots, j_d) \\
		& \qquad  [b(X_i^N(s),X_{j_2}^N(s),\dotsc,X_{j_d}^N(s)) - b(X_i(s),X_{j_2}(s),\dotsc,X_{j_d}(s))] \Big]^2. \notag	
	\end{align*}
	From the Lipschitz property of $b$, the definition of $\Smc_i^N$ and Cauchy-Schwarz inequality we have
	\begin{align*}		
		R_i^{N,1}(s) & \le \Ebf \Big[ \one_{\{D_i \ge d-1\}} \sum_{(j_2,\dotsc,j_d) \in \Smc_i^N} \alpha^N(i; j_2, j_3, \ldots, j_d) \\
		& \qquad  (|X_i^N(s)-X_i(s)| + |X_{j_2}^N(s)-X_{j_2}(s)| + \dotsb + |X_{j_d}^N(s)-X_{j_d}(s)|) \Big]^2 \notag \\
		& = \Ebf \Big[ \one_{\{D_i \ge d-1\}} \Big( |X_i^N(s)-X_i(s)| +  (d-1) \sum_{j \in [N], j \ne i} \frac{\xi_{ij}}{D_i} |X_j^N(s)-X_j(s)| \Big) \Big]^2 \\
		& \le 2 \Ebf |X_i^N(s)-X_i(s)|^2 + 2(d-1)^2 \Ebf \Big[ \one_{\{D_i \ge d-1\}} \sum_{j \in [N], j \ne i} \frac{\xi_{ij}}{D_i} |X_j^N(s)-X_j(s)|^2 \Big].
	\end{align*}	
	From Cauchy-Schwarz inequality $(\sum a_ib_i)^2 \le (\sum a_i) (\sum a_i b_i^2)$ for non-negative $a_i$'s we have
	\begin{align*}
		R_i^{N,2}(s) & \le \Ebf \Big\{ \Big[ \sum_{(j_2,\dotsc,j_d) \in \Smc_i^N} \one_{\{D_{j_2} \ge d-1\}} \alpha^N(j_2; i, j_3, \ldots, j_d) \Big] \Big[ \sum_{(j_2,\dotsc,j_d) \in \Smc_i^N} \one_{\{D_{j_2} \ge d-1\}} \alpha^N(j_2; i, j_3, \ldots, j_d) \\
		& \qquad  [b(X_i^N(s),X_{j_2}^N(s),\dotsc,X_{j_d}^N(s)) - b(X_i(s),X_{j_2}(s),\dotsc,X_{j_d}(s))]^2 \Big] \Big\} \notag \\
		& = \Ebf \Big[ \sum_{(j_2,\dotsc,j_d) \in \Smc_i^N} \one_{\{D_{j_2} \ge d-1\}} \alpha^N(j_2; i, j_3, \ldots, j_d) \\
		& \qquad  [b(X_i^N(s),X_{j_2}^N(s),\dotsc,X_{j_d}^N(s)) - b(X_i(s),X_{j_2}(s),\dotsc,X_{j_d}(s))]^2 \Big] \notag \\
		& \quad + \Ebf \Big\{ \Big[ \sum_{j \in [N], j \ne i} \one_{\{D_j \ge d-1\}} \frac{\xi_{ji}}{D_j} - 1 \Big] \Big[ \sum_{(j_2,\dotsc,j_d) \in \Smc_i^N} \one_{\{D_{j_2} \ge d-1\}} \alpha^N(j_2; i, j_3, \ldots, j_d) \\
		& \qquad  [b(X_i^N(s),X_{j_2}^N(s),\dotsc,X_{j_d}^N(s)) - b(X_i(s),X_{j_2}(s),\dotsc,X_{j_d}(s))]^2 \Big] \Big\} \notag \\
		& \doteq R_i^{N,3}(s) + R_i^{N,4}(s),
	\end{align*}
	where the equality follows by adding and subtracting one in the first term.
	From the Lipschitz property of $b$, the definition of $\Smc_i^N$ and the exchangeability property \eqref{eq:exchangeability} we have
	\begin{align*}
		R_i^{N,3}(s) & \le d\ \Ebf \Big[ \sum_{(j_2,\dotsc,j_d) \in \Smc_i^N} \one_{\{D_{j_2} \ge d-1\}} \alpha^N(j_2; i, j_3, \ldots, j_d) \notag \\
		& \qquad (|X_i^N(s)-X_i(s)|^2 + |X_{j_2}^N(s)-X_{j_2}(s)|^2 + \dotsb + |X_{j_d}^N(s)-X_{j_d}(s)|^2) \Big] \notag \\
		& = d\ \Ebf \Big[ \sum_{(j_2,\dotsc,j_d) \in \Smc_i^N} \one_{\{D_i \ge d-1\}} \alpha^N(i; j_2, j_3, \ldots, j_d) \notag \\
		& \qquad (|X_i^N(s)-X_i(s)|^2 + |X_{j_2}^N(s)-X_{j_2}(s)|^2 + \dotsb + |X_{j_d}^N(s)-X_{j_d}(s)|^2) \Big] \notag \\
		& \le d\ \Ebf |X_i^N(s)-X_i(s)|^2 + d(d-1) \Ebf \Big[ \one_{\{D_i \ge d-1\}} \sum_{j \in [N], j \ne i} \frac{\xi_{ij}}{D_i} |X_j^N(s)-X_j(s)|^2 \Big]. 
	\end{align*}
	From the fact that $\|b\|_\infty \le 1$ we have
	\begin{align*}
		R_i^{N,4}(s) & \le \Ebf \Big\{ \Big| \sum_{j \in [N], j \ne i} \one_{\{D_j > 0\}} \frac{\xi_{ji}}{D_j} - 1 \Big| \Big[ 4 \sum_{(j_2,\dotsc,j_d) \in \Smc_i^N} \one_{\{D_{j_2} \ge d-1\}} \alpha^N(j_2; i, j_3, \ldots, j_d) \Big] \Big\} \notag \\
		& \le 4 \Ebf \Big\{ \Big| \sum_{j \in [N], j \ne i} \one_{\{D_j > 0\}} \frac{\xi_{ji}}{D_j} - 1 \Big| \sum_{j \in [N], j \ne i} \one_{\{D_j > 0\}} \frac{\xi_{ji}}{D_j} \Big\} \notag \\
		& \le \kappa_3 \left( \frac{1}{Np_N} + e^{-Np_N} \right)^{1/2}.
	\end{align*}
	where the last inequality follows from Lemma \ref{lem:prep_2} and Condition \ref{cond:errg1}.
	Combining the above estimates on $R_i^{N,k}(s)$ for $k=1,2,3,4$ gives the claim \eqref{eq:CN_CN1} when $m=2$.
	
	Now using the exchangeability property:
	\begin{align*}
		\Lmc(\xi_{ij}, D_i, X_j^N(s), X_j(s)) & = \Lmc(\xi_{ji}, D_j, X_i^N(s), X_i(s)), \quad i \ne j,
	\end{align*}	
	we have for $m=1,2$,
	\begin{align*}
		& \quad \Ebf \Big[ \one_{\{D_i \ge d-1\}} \sum_{j \in [N], j \ne i} \frac{\xi_{ij}}{D_i} |X_j^N(s)-X_j(s)|^m \Big] \\
		& = \Ebf \Big[ \sum_{j \in [N], j \ne i} \one_{\{D_j \ge d-1\}} \frac{\xi_{ji}}{D_j} |X_i^N(s)-X_i(s)|^m \Big] \\
		& \le \Ebf \Big[ \Big( \sum_{j \in [N], j \ne i} \one_{\{D_j > 0\}}  \frac{\xi_{ji}}{D_j} - 1 \Big) |X_i^N(s)-X_i(s)|^m \Big] + \Ebf |X_i^N(s)-X_i(s)|^m \\
		& \le \Big[ \Ebf \Big( \sum_{j \in [N], j \ne i} \one_{\{D_j > 0\}}  \frac{\xi_{ji}}{D_j} - 1 \Big)^2 \Ebf |X_i^N(s)-X_i(s)|^{2m} \Big]^{1/2} + \Ebf |X_i^N(s)-X_i(s)|^m \\
		& \le \kappa_4 \Big(\frac{1}{Np_N} + e^{-Np_N}\Big)^{1/2} + \Ebf |X_i^N(s)-X_i(s)|^m, 
	\end{align*}	
	where the second inequality follows from Cauchy-Schwarz inequality and the last line follows from Lemmas \ref{lem:prep_2} and \ref{lem:momentbd}.
	Combining this, \eqref{eq:CN_CN1} with the fact that $|X_i^N(s)-X_i(s)| \le |X_i^N(s)-X_i(s)|^2$ gives
	\begin{equation}
		\label{eq:npn_difference_1}
		\Ebf |C_i^N(s)-C_i^{N,1}(s)| + \Ebf |C_i^N(s)-C_i^{N,1}(s)|^2 \le \kappa_5 \Ebf |X_i^N(s)-X_i(s)|^2 + \kappa_5 \left(\frac{1}{Np_N} + e^{-Np_N}\right)^{1/2}.
	\end{equation}
	
	Next we consider $|C_i^{N,1}(s)-C_i^{N,2}(s)|^2$.
	From the inequality $(a+b)^2 \le 2a^2+2b^2$, it follows  that
	\begin{align}
		\left( \Ebf |C_i^{N,1}(s)-C_i^{N,2}(s)| \right)^2 \le \Ebf |C_i^{N,1}(s)-C_i^{N,2}(s)|^2
		 \le 2U_s^A + 2(d-1)^2V_s^A \le \frac{\kappa_6}{Np_N} + \frac{\kappa_6}{(Np_N)^2},\label{eq:npn_difference_2}
	\end{align}
	where $U_s^A$ and $V_s^A$ were introduced in \eqref{eq:prep_new_i} and \eqref{eq:prep_new_j} and the last inequality is from Lemma
	\ref{lem:prep_new}.

	Finally we consider $|C_i^{N,2}(s)-C_i(s)|^2$.
	Note that $C_i^{N,2}(s)$ can be rewritten as
	\begin{align*}
		C_i^{N,2}(s) & = \one_{\{D_i<d-1\}} \bbar_i((X_k^N(t))_{k \in [N]},(\xi_{kl})_{k,l \in [N]}) \\
		& \quad + \one_{\{D_i \ge d-1\}} \frac{C_i(s)}{d} + (d-1) \sum_{j \in [N], j \ne i} \one_{\{D_{j} \ge d-1\}} \frac{\xi_{ji}}{D_j} \frac{C_i(s)}{d} \\
		& \quad + \sum_{j \in [N], j \ne i} \one_{\{D_{j}<d-1\}}\xi_{ij} \bbar_{ij}((X_k^N(t))_{k \in [N]},(\xi_{kl})_{k,l \in [N]}).
	\end{align*}
	Using the Cauchy-Schwarz inequality and the fact that $0 \le \frac{C_i(s)}{d} \le 1$, we have
	\begin{align}
		& \Ebf |C_i^{N,2}(s)-C_i(s)|^2 \notag \\
		& \le 5\Ebf \Big[ \one_{\{D_i<d-1\}} (D_i+1) \Big]^2 + 5\Ebf \Big[ \one_{\{D_i<d-1\}} \frac{C_i(s)}{d} \Big]^2 + 5\Ebf \left[(d-1)\sum_{j \in [N], j \ne i} \one_{\{0 < D_{j} < d-1\}} \frac{\xi_{ji}}{D_j} \frac{C_i(s)}{d}\right]^2 \notag \\
		& \quad + 5\Ebf \Big[ \frac{(d-1)C_i(s)}{d} \Big| \sum_{j \in [N], j \ne i} \one_{\{D_{j} > 0\}} \frac{\xi_{ji}}{D_j} -1 \Big| \Big]^2 + 5\Ebf \Big[ \sum_{j \in [N], j \ne i} \one_{\{D_{j}<d-1\}}\xi_{ij} (D_i+1) \Big]^2 \notag \\
		& \le 5(d^2+1) \Pbf(D_i < d-1) + 5(d-1)^2\Ebf \left[\sum_{j \in [N], j \ne i} \one_{\{0 < D_{j} < d-1\}} \frac{\xi_{ji}}{D_j}\right]^2 \notag \\
		& \quad + 5(d-1)^2 \Ebf \Big[ \sum_{j \in [N], j \ne i} \one_{\{D_{j} > 0\}} \frac{\xi_{ji}}{D_j} -1 \Big]^2 + 5\Ebf \Big[ \sum_{j \in [N], j \ne i} \one_{\{D_{j}<d-1\}}\xi_{ij} (D_i+1) \Big]^2. 
		\label{eq:npn_difference_3_temp}
	\end{align}	
	Note that on the right hand side of \eqref{eq:npn_difference_3_temp}, the second term can be bounded by the last term as follows
	\begin{equation*}
		\Ebf \Big[\sum_{j \in [N], j \ne i} \one_{\{0 < D_{j} < d-1\}} \frac{\xi_{ji}}{D_j}\Big]^2 \le \Ebf \Big[\sum_{j \in [N], j \ne i} \one_{\{D_{j} < d-1\}} \xi_{ji}\Big]^2 \le \Ebf \Big[ \sum_{j \in [N], j \ne i} \one_{\{D_{j}<d-1\}}\xi_{ij} (D_i+1) \Big]^2.
	\end{equation*}
	For the last term in \eqref{eq:npn_difference_3_temp} we have
	\begin{align*}
		\Ebf \Big[ \sum_{j \in [N], j \ne i} \one_{\{D_{j}<d-1\}}\xi_{ij} (D_i+1) \Big]^2
		& \le \Ebf \Big\{ \Big[ \sum_{j \in [N], j \ne i} \one_{\{D_{j}<d-1\}}\xi_{ij} (D_i+1)^2 \Big] \Big[ \sum_{j \in [N], j \ne i} \xi_{ij} \Big] \Big\} \\
		& = \sum_{j \in [N], j \ne i} \Ebf \Big[ \one_{\{D_{j}<d-1\}}\xi_{ij} (D_i+1)^2 D_i \Big] \\
		& = \sum_{j \in [N], j \ne i} \Ebf \Big[ \one_{\{D_{j}-\xi_{ij}+1<d-1\}} (D_i-\xi_{ij}+2)^2 (D_i-\xi_{ij}+1) \Big] p_N \\
		& \le \kappa_7 (N-1) \Pbf(D_i<d) (Np_N+1)^3p_N,
	\end{align*}
	where the first inequality follows from Cauchy-Schwarz inequality, the second equality follows by conditioning on $\xi_{ij}=1$, and the last inequality follows from independence, Condition \ref{cond:errg1} and moment estimates of binomial random variables.	
	Furthermore, note that
	\begin{align}
 		\Pbf(D_i<d) 
		& = \sum_{k=0}^{d-1}\binom{N-1}{k} p_N^k (1-p_N)^{N-1-k}  \notag \\
		& \le \kappa_8 (1-p_N)^{N-d} \left[ 1 + Np_N + \dotsb + (Np_N)^{d-1} \right] \notag \\
		& \le \kappa_9 [1+(Np_N)^{d-1}] e^{-(N-d)p_N}.
		\label{eq:CN_bd_2}
	\end{align}	
	Combining above four estimates with Lemma \ref{lem:prep_2} gives
	\begin{equation}
		\label{eq:npn_difference_3}
		\left( \Ebf |C_i^{N,2}(s)-C_i(s)| \right)^2 \le \Ebf |C_i^{N,2}(s)-C_i(s)|^2 \le \kappa_0 [1+(Np_N)^{d+3}] e^{-Np_N} + \kappa_0 \Big(\frac{1}{Np_N} + e^{-Np_N}\Big).
	\end{equation}	
	Combining \eqref{eq:npn_2}, \eqref{eq:npn_1}, \eqref{eq:npn_3}, \eqref{eq:npn_difference_1}, \eqref{eq:npn_difference_2}, \eqref{eq:npn_difference_3} and Condition \ref{cond:errg1} gives us
	\begin{align*}
		\max_{i \in [N]} \sqrt{Np_N} \Ebf \left\|X_i^N - X_i\right\|_{*,t}^2 & \le \kappa \int_0^t \max_{i \in [N]} \sqrt{Np_N} \Ebf \left\|X_i^N - X_i\right\|_{*,s}^2 \, ds + \kappa.
	\end{align*}
	Part (i) of the theorem now follows from Gronwall's lemma.
	
	The proof of propagation of chaos property as stated in Theorem~\ref{thm:npn_rate} (ii) follows now from standard arguments (cf.\ \cite{Sznitman1991}), and hence is omitted.
	Also, having proved Theorem~\ref{thm:npn_rate} (i), the proof of convergence of local occupancy measures as stated in Theorem~\ref{thm:npn_rate} (iii) can be established using similar arguments as in \cite[Corollary 3.3]{BhamidiBudhirajaWu2016}.
\end{proof}

We now complete the proof of Theorem \ref{thm:npn_rate_quench}.

\begin{proof}[Proof of Theorem~\ref{thm:npn_rate_quench}]
In order to prove the theorem it suffices, in view of Theorems \ref{th:deterministic} and~\ref{th:tagged}, to show that  if  $\{G_N\}$ satisfies Condition~\ref{cond:errg2}, then it satisfies Condition~\ref{cond-reg} a.s.

Using the Chernoff inequality (cf. \cite[Theorem 2.4]{CL06}), it follows that for every $x \ge 0$ and $N \in \Nmb$,
\begin{equation*}
	\Pbf(|D_i^N-\Ebf D_i^N| \ge x) \le 2 \exp \left\{ -\frac{x^2}{2\Ebf D_i^N + 2x/3}\right\}.
\end{equation*}
Let $k(N)\doteq Np_N/\ln(N).$ Note that by Condition \ref{cond:errg2}, $k(N) \to \infty$ as $N\to\infty$.
Since $\Ebf D_i^N = (N-1)p_N$ taking $x=x(N) =\ln(N)(k(N))^{3/4}$ in the above expression yields, for some $\kappa_1 \in (0,\infty)$,
\begin{equation}\label{eq:chernoff}
\begin{split}
\Pbf(|D_i^N-Np_N| \ge x(N)) & \le \Pbf(|D_i^N-\Ebf D_i^N| \ge x(N)-p_N) \\
	& \le 2 \exp \Big\{ -\frac{(x(N)-p_N)^2}{2(N-1)p_N + 2(x(N)-p_N)/3}\Big\}\\
	& \le \kappa_1\exp \Big\{ -\kappa_1\frac{(x(N))^2}{Np_N}\Big\},
\end{split}
\end{equation}
for sufficiently large $N$.
Thus
\begin{equation}
	\Pbf\left(\bigcup_{i\in [N]}\left\{|D_i^N-Np_N| \ge x(N)\right\}\right) \le \kappa_1N\exp \Big\{ -\kappa_1\frac{(x(N))^2}{Np_N}\Big\}.
	\label{eq:uninbd}
\end{equation}
From the choice ot $x(N)$, we have  $(x(N))^2/[Np_N\ln(N)]\to \infty$, as $N\to\infty$.
Therefore, the right side of~\eqref{eq:uninbd} is summable over $N$.
From  Borel--Cantelli lemma we conclude a.s., for all sufficiently large~$N$,
\begin{equation*}
	|D_i^N-Np_N| \le x(N), \quad i \in [N]
\end{equation*}
and therefore for all such $N$
\begin{equation}
		\label{eq:degree_key}
		Np_N - x(N) \le d_{\min}(G_N) \le d_{\max}(G_N) \le Np_N + x(N)
\end{equation}
Finally, observe that
$$\frac{x(N)}{Np_N} = \frac{\ln(N) (k(N))^{3/4}}{k(N) \ln(N)}= \frac{1}{(k(N))^{1/4}}\to 0 \quad\mbox{as }N\to\infty.$$
Combining the two displays, $d_{\min}(G_N)  \to \infty$ and
$$\frac{d_{\max}(G_N)-d_{\min}(G_N)}{d_{\min}(G_N)} = \frac{2x(N)}{Np_N-x(N)}\to 0,$$
as $N\to\infty$. 
This together with Remark \ref{rmk:weaker_condition} shows that Condition~\ref{cond-reg} holds for $\{G_N\}$ a.s., completing the proof of Theorem \ref{thm:npn_rate_quench}.
\end{proof}
We now complete the proof of Lemma~\ref{lem:prep_new}. We begin with the following lemma from \cite{BhamidiBudhirajaWu2016}.

\begin{lemma}[{\cite[Lemma~5.1]{BhamidiBudhirajaWu2016}}]
	\label{lem:prep_1}
	Let $X$ be a Binomial random variable with number of trials $N$ and probability of success $p$.
	Let $q \doteq 1 - p$.
	Then for each $m \in \Nmb$,
	\begin{equation*}
		\E \left[ \one_{\{X>0\}} \frac{1}{(2X)^m} \right] \le \E \frac{1}{(X+1)^m} \le \frac{m^m}{(N+1)^mp^m}.
	\end{equation*}
\end{lemma}

\begin{proof}[Proof of Lemma~\ref{lem:prep_new}]
	As before, we will omit the superscript in $\xi_{ij}$'s and $D_i$'s for notational convenience.
	We first show \eqref{eq:prep_new_i}.
	From the independence between $\{X_i\}$ and $\{\xi_{ij}\}$ it follows that
	\begin{align}
		U_s^A & = \sum_{(j_2,\dotsc,j_d) \in \Smc_i^N} \sum_{(k_2,\dotsc,k_d) \in \Smc_i^N} \Ebf \left[ \one_{\{D_i \ge d-1\}} \alpha^N(i; j_2, j_3, \ldots, j_d) \alpha^N(i; k_2, k_3, \ldots, k_d)\right] \notag \\
		& \quad \Ebf \left[ \left(b(X_i(s),X_{j_2}(s),\dotsc,X_{j_d}(s)) - \frac{C_i(s)}{d} \right) \left(b(X_i(s),X_{k_2}(s),\dotsc,X_{k_d}(s)) - \frac{C_i(s)}{d} \right) \right]. \notag
	\end{align}
	Noting that
	\begin{equation}
		\label{eq:prep_new_simplify}
		\Ebf \left[ \left(b(X_i(s),X_{j_2}(s),\dotsc,X_{j_d}(s)) - \frac{C_i(s)}{d} \right) \left(b(X_i(s),X_{k_2}(s),\dotsc,X_{k_d}(s)) - \frac{C_i(s)}{d} \right) \right]=0
	\end{equation}
	when $(i,j_2,k_2,\dotsc,j_d,k_d)$ are distinct, we have
	\begin{align}
		U_s^A & = \sum \Ebf \left[ \one_{\{D_i \ge d-1\}} \alpha^N(i; j_2, j_3, \ldots, j_d) \alpha^N(i; k_2, k_3, \ldots, k_d)\right] \notag \\
		& \quad \Ebf \left[ \left(b(X_i(s),X_{j_2}(s),\dotsc,X_{j_d}(s)) - \frac{C_i(s)}{d} \right) \left(b(X_i(s),X_{k_2}(s),\dotsc,X_{k_d}(s)) - \frac{C_i(s)}{d} \right) \right] \notag \\
		& \le \Ebf \left[ \sum \one_{\{D_i \ge d-1\}} \alpha^N(i; j_2, j_3, \ldots, j_d) \alpha^N(i; k_2, k_3, \ldots, k_d)\right], \label{eq:prep_new_2}
	\end{align}
	where the summation is taken over the collection $\hat \Smc_i^N$ defined in \eqref{eq:prep_quench_3}
	%
	and the inequality follows since $0 \le b \le 1$ and $0 \le \frac{C_i(s)}{d} \le 1$.
	As noted in \eqref{eq:bdonbindiff},  the total number of combinations in \eqref{eq:prep_quench_3} such that $(\xi_{ij_2}\xi_{ij_3}\dotsm\xi_{ij_d})(\xi_{ik_2}\xi_{ik_3}\dotsm\xi_{ik_d})=1$ is no more than
	$ \kappa_1 D_i^{2d-3}$ and thus
	we can bound \eqref{eq:prep_new_2} by
	\begin{align*}
		\Ebf \left[ \one_{\{D_i \ge d-1\}} \frac{\kappa_1 D_i^{2d-3}}{D_i^2(D_i-1)^2\dotsm(D_i-d+2)^2} \right] \le \kappa_2 \Ebf \left[ \one_{\{D_i > 0\}} \frac{1}{D_i} \right] \le \frac{2\kappa_2}{Np_N},
	\end{align*}
	where the last inequality uses Lemma \ref{lem:prep_1}.
	This gives the first inequality in Lemma \ref{lem:prep_new}.
	
	Next we show the second inequality in Lemma \ref{lem:prep_new}.
	From the independence between $\{X_i\}$ and $\{\xi_{ij}\}$ and \eqref{eq:prep_new_simplify} it follows from the same argument used for \eqref{eq:prep_new_2} that 
	\begin{equation}
		V_s^A \le \Ebf \left[ \sum \one_{\{D_{j_2} \ge d-1\}} \one_{\{D_{k_2} \ge d-1\}} \alpha^N(j_2; i, j_3, \ldots, j_d) \alpha^N(k_2; i, k_3, \ldots, k_d)\right], \label{eq:prep_new_4}
	\end{equation}
	where the summation is taken over $\hat \Smc_i^N$ defined in \eqref{eq:prep_quench_3}.
	As noted in \eqref{eq:d2d3bd},  for fixed $(j_2,k_2) \in \Smcbar_i$ with $\Smcbar_i$ as in~\eqref{eq:smcbar}, the total number of combinations in $\hat \Smc_i^N$  such that 
	\[(\xi_{j_2i}\xi_{j_2j_3}\dotsm\xi_{j_2j_d})(\xi_{k_2i}\xi_{k_2k_3}\dotsm\xi_{k_2k_d})=1\] is no more than
	$\kappa_3 (D_{j_2}^{d-3}D_{k_2}^{d-2} + D_{j_2}^{d-2}D_{k_2}^{d-3})$
	we can bound \eqref{eq:prep_new_4} by
	\begin{align}
		& \Ebf \left[ \sum_{(j_2,k_2) \in \Smcbar_i} \one_{\{D_{j_2} \ge d-1\}} \one_{\{D_{k_2} \ge d-1\}} \frac{\kappa_3 (D_{j_2}^{d-3}D_{k_2}^{d-2} + D_{j_2}^{d-2}D_{k_2}^{d-3})\xi_{j_2i}\xi_{k_2i}}{D_{j_2}(D_{j_2}-1)\dotsm(D_{j_2}-d+2)D_{k_2}(D_{k_2}-1)\dotsm(D_{k_2}-d+2)}\right] \notag \\
		& \le \kappa_4 \sum_{(j_2,k_2) \in \Smcbar_i} \Ebf \left[ \one_{\{D_{j_2} \ge d-1\}} \one_{\{D_{k_2} \ge d-1\}} \left( \frac{\xi_{j_2i}\xi_{k_2i}}{D_{j_2}^2D_{k_2}} + \frac{\xi_{j_2i}\xi_{k_2i}}{D_{j_2}D_{k_2}^2} \right) \right] \notag \\
		& = 2 \kappa_4 \sum_{(j,k) \in \Smcbar_i} \Ebf \left[ \one_{\{D_{j} \ge d-1\}} \one_{\{D_{k} \ge d-1\}} \frac{\xi_{ji}\xi_{ki}}{D_{j}^2D_{k}} \right]. \label{eq:prep_new_5}
	\end{align}
	Now for $(j,k) \in \Smcbar_i$ with $j \ne k$, we have
	\begin{align*}
		& \Ebf \left[ \one_{\{D_{j} \ge d-1\}} \one_{\{D_{k} \ge d-1\}} \frac{\xi_{ji}\xi_{ki}}{D_{j}^2D_{k}} \right] \\
		& = \Ebf \left[ \one_{\{\xi_{jk}=1\}} \one_{\{D_{j} \ge d-1\}} \one_{\{D_{k} \ge d-1\}} \frac{\xi_{ji}\xi_{ki}}{D_{j}^2D_{k}} \right] + \Ebf \left[ \one_{\{\xi_{jk}=0\}} \one_{\{D_{j} \ge d-1\}} \one_{\{D_{k} \ge d-1\}} \frac{\xi_{ji}\xi_{ki}}{D_{j}^2D_{k}} \right] \\
		& \le \Ebf \left[ \frac{\xi_{ji}\xi_{ki}}{(D_{j}-\xi_{jk}+1)^2(D_{k}-\xi_{jk}+1)} \right] + \Ebf \left[ \one_{\{D_j-\xi_{jk} > 0\}} \one_{\{D_{k}-\xi_{jk} > 0\}} \frac{\xi_{ji}\xi_{ki}}{(D_{j}-\xi_{jk})^2(D_{k}-\xi_{jk})} \right] \\
		& = \Ebf \left[ \frac{\xi_{ji}}{(D_{j}-\xi_{jk}+1)^2} \right] \Ebf \left[ \frac{\xi_{ki}}{D_{k}-\xi_{jk}+1} \right]
		+ \Ebf \left[ \one_{\{D_j-\xi_{jk} > 0\}} \frac{\xi_{ji}}{(D_{j}-\xi_{jk})^2} \right] \Ebf \left[ \one_{\{D_{k}-\xi_{jk} > 0\}} \frac{\xi_{ki}}{D_{k}-\xi_{jk}} \right],
	\end{align*}
	where the last equality follows from independence between $(\xi_{ji},D_j-\xi_{jk})$ and $(\xi_{ki},D_k-\xi_{jk})$.
	Using exchangeability and Lemma \ref{lem:prep_1} we have
	\begin{align*}
		\Ebf \left[ \frac{\xi_{ji}}{(D_{j}-\xi_{jk}+1)^2} \right] & = \frac{1}{N-2} \sum_{l \in [N], l \ne j,k} \Ebf \left[ \frac{\xi_{jl}}{(D_{j}-\xi_{jk}+1)^2} \right] = \frac{1}{N-2} \Ebf \left[ \frac{D_j-\xi_{jk}}{(D_{j}-\xi_{jk}+1)^2} \right] \\
		& \le \frac{1}{N-2} \Ebf \left[ \frac{1}{D_{j}-\xi_{jk}+1} \right] \le \frac{1}{(N-2)(N-1)p_N}.
	\end{align*}
	Similarly one can verify that
	\begin{align*}
		 &\Ebf \left[ \frac{\xi_{ki}}{D_{k}-\xi_{jk}+1} \right] \le \frac{1}{N-2}, \\
		 &\Ebf \left[ \one_{\{D_j-\xi_{jk} > 0\}} \frac{\xi_{ji}}{(D_{j}-\xi_{jk})^2} \right] \le \frac{4}{(N-2)(N-1)p_N}, \quad
		 \Ebf \left[ \one_{\{D_{k}-\xi_{jk} > 0\}} \frac{\xi_{ki}}{D_{k}-\xi_{jk}} \right] \le \frac{1}{N-2}.
	\end{align*}
	Combining these gives us 
	\begin{equation*}
		\Ebf \left[ \one_{\{D_{j} \ge d-1\}} \one_{\{D_{k} \ge d-1\}} \frac{\xi_{ji}\xi_{ki}}{D_{j}^2D_{k}} \right] \le \frac{5}{(N-2)^2(N-1)p_N}, \mbox{ when } j \ne k.
	\end{equation*}
	Also note that the summation in \eqref{eq:prep_new_5} when $j=k$ is 
	\begin{equation*}
		\sum_{j=1,j \ne i}^N \Ebf \left[ \one_{\{D_{j} \ge d-1\}} \frac{\xi_{ji}}{D_{j}^3} \right] = \sum_{j=1,j \ne i}^N \Ebf \left[ \one_{\{D_{i} \ge d-1\}} \frac{\xi_{ij}}{D_{i}^3} \right] = \Ebf \left[ \one_{\{D_{i} \ge d-1\}} \frac{1}{D_i^2} \right] \le \frac{4}{(Np_N)^2},
	\end{equation*}
	where the first equality uses exchangeability and the inequality uses Lemma \ref{lem:prep_1}.
	Combining these two estimates with \eqref{eq:prep_new_5} gives
	\begin{equation*}
		V_s^A \le \kappa_5  \frac{N^2}{(N-2)^2(N-1)p_N} + \kappa_5 \frac{1}{(Np_N)^2}\le \frac{\kappa_6}{Np_N} + \frac{\kappa_6}{(Np_N)^2}
	\end{equation*}
	for some $\kappa_5, \kappa_6 \in (0,\infty)$.
	This completes the proof of Lemma~\ref{lem:prep_new}.
\end{proof}

Finally we complete the proof of Lemma \ref{lem:momentbd}.
\begin{proof}[Proof of Lemma \ref{lem:momentbd}]
	As before, we will omit the superscript in $\xi_{ij}$'s and $D_i$'s for notational convenience.
	Fix $i\in\N$.
	From \eqref{eq:X_i_n} and \eqref{eq:limit-tagged}, using Cauchy--Schwarz and Doob's inequalities we have for any fixed $t\in [0,T]$	
	\begin{align}
		\Ebf \left\|X_i^N - X_i\right\|_{*,t}^4 & \le 
		\kappa_1 \int_0^t \Ebf |X_i^N(s) - X_i(s)|^4 \, ds + \kappa_1 \int_0^t \Ebf |C_i^N(s)-C_i(s)|^2 \, ds + \kappa_1 \int_0^t \Ebf |C_i^N(s)-C_i(s)|^4 \, ds. \label{eq:momentbd1}
	\end{align}
	Recall the definition of $C_i^N(s)$ and $C_i(s)$ from~\eqref{eq:cn} and \eqref{eq:limit-tagged}.
	From the bound $\|b\|_\infty \le 1$ and \eqref{eq:bbar}, for $s \in [0,T]$ we have $|C_i(s)| \le d$ and
	\begin{align}
		\Ebf |C_i^N(s)|^4 & \le \Ebf \left| \one_{\{D_i<d-1\}}(D_i+1) + 1 + (d-1) \sum_{j_2 \in [N], j_2 \ne i} \one_{\{D_{j_2} \ge d-1\}} \frac{\xi_{j_2i}}{D_{j_2}} \right. \notag \\
		& \left. \qquad + \sum_{j_2 \in [N], j_2 \ne i} \one_{\{D_{j_2} < d-1\}} \xi_{ij_2}(D_i+1) \right|^4 \notag \\
		& \le \kappa_2 + \kappa_2 \Ebf \Big[ \sum_{j_2 \in [N], j_2 \ne i} \one_{\{D_{j_2} \ge d-1\}} \frac{\xi_{j_2i}}{D_{j_2}} \Big]^4 + \kappa_2 \Ebf\Big[\sum_{j_2 \in [N], j_2 \ne i} \one_{\{D_{j_2} < d-1\}} \xi_{ij_2}(D_i+1)\Big]^4. \label{eq:CN_bd}
	\end{align}
	Here the second term on the right hand side can be written as
	\begin{align*}
		& \kappa_2 \Ebf \Big[ \sum_{j_2 \in [N], j_2 \ne i} \one_{\{D_{j_2} \ge d-1, D_i > 0\}} \frac{D_i}{D_{j_2}} \frac{\xi_{j_2i}}{D_i} \Big]^4 \\
		& \le \kappa_2 \Ebf \Big\{ \Big[ \sum_{j_2 \in [N], j_2 \ne i} \one_{\{D_{j_2} \ge d-1, D_i > 0\}} \left(\frac{D_i}{D_{j_2}}\right)^4 \frac{\xi_{j_2i}}{D_i} \Big] \Big[ \sum_{j_2 \in [N], j_2 \ne i} \one_{\{D_{j_2} \ge d-1, D_i > 0\}} \frac{\xi_{j_2i}}{D_i} \Big]^3 \Big\} \\
		& \le \kappa_2 \Ebf \sum_{j_2 \in [N], j_2 \ne i} \one_{\{D_{j_2} \ge d-1\}} \frac{D_i^3 \xi_{j_2i}}{D_{j_2}^4}\\
		& = \kappa_2 \sum_{j_2 \in [N], j_2 \ne i} \Ebf \Big[ \one_{\{D_{j_2}-\xi_{j_2i}+1 \ge d-1\}} \frac{(D_i-\xi_{j_2i}+1)^3 }{(D_{j_2}-\xi_{j_2i}+1)^4} \Big] p_N \\
		& = \kappa_2 \sum_{j_2 \in [N], j_2 \ne i} \Ebf \Big[ \one_{\{D_{j_2}-\xi_{j_2i}+1 \ge d-1\}} \frac{1}{(D_{j_2}-\xi_{j_2i}+1)^4} \Big] \Ebf \Big[ D_i-\xi_{j_2i}+1 \Big]^3 p_N \\
		& \le \kappa_3 (N-1) \frac{1}{(N-1)^4p_N^4}(Np_N+1)^3p_N \le \kappa_4,
	\end{align*}
	where the second line uses Holder's inequality, the fourth line follows by conditioning on $\xi_{j_2i}=1$, the fifth line follows from independence, and the last line uses Lemma \ref{lem:prep_1} and moment estimates of binomial random variables.
	Following the similar argument, we can write the last term in \eqref{eq:CN_bd} as
	\begin{align*}
		& \kappa_2 \Ebf\Big[\sum_{j_2 \in [N], j_2 \ne i} \one_{\{D_{j_2} < d-1, D_i > 0\}} D_i(D_i+1) \frac{\xi_{ij_2}}{D_i} \Big]^4 \notag \\
		& \le \kappa_2 \Ebf \Big\{ \Big[ \sum_{j_2 \in [N], j_2 \ne i} \one_{\{D_{j_2} < d-1, D_i > 0\}} D_i^4(D_i+1)^4 \frac{\xi_{ij_2}}{D_i} \Big] \Big[ \sum_{j_2 \in [N], j_2 \ne i} \one_{\{D_{j_2} < d-1, D_i > 0\}} \frac{\xi_{ij_2}}{D_i} \Big]^3 \Big\} \notag \\
		& \le \kappa_2 \Ebf \sum_{j_2 \in [N], j_2 \ne i} \one_{\{D_{j_2} < d-1\}} D_i^3(D_i+1)^4\xi_{ij_2} \notag \\
		& = \kappa_2 \sum_{j_2 \in [N], j_2 \ne i} \Ebf \Big[\one_{\{D_{j_2}-\xi_{ij_2}+1 < d-1\}} (D_i-\xi_{ij_2}+1)^3 (D_i-\xi_{ij_2}+2)^4 \Big]p_N \notag \\
		& = \kappa_2 \sum_{j_2 \in [N], j_2 \ne i} \Ebf \Big[\one_{\{D_{j_2}-\xi_{ij_2}+1 < d-1\}} \Big] \Ebf \Big[ (D_i-\xi_{ij_2}+1)^3 (D_i-\xi_{ij_2}+2)^4 \Big] p_N \notag \\
		& \le \kappa_5 (N-1) \Pbf (D_{i}< d) (Np_N+1)^7 p_N. 
	\end{align*}
	Combining above three estimates with \eqref{eq:CN_bd_2} and using Condition \ref{cond:errg1}, we have
		$\Ebf |C_i^N(s)|^4 \le \kappa_6$.
	It then follows from \eqref{eq:momentbd1} that 
	\begin{equation*}
		\Ebf \left\|X_i^N - X_i\right\|_{*,t}^4 \le \kappa_7 \int_0^t \Ebf \|X_i^N - X_i\|_{*,s}^4 \, ds + \kappa_7.
	\end{equation*}
	The result then follows from Gronwall's inequality.
\end{proof}

\section{Conclusion}\label{sec:conclusion}
We have considered the JSQ($d$) policy in large-scale systems where the servers communicate with their neighbors and the neighborhood relationships are described in terms of a suitable graph.
We have developed sufficient criteria for arbitrary graph sequences so that asymptotically the evolution of the occupancy process on any finite time interval is indistinguishable from that for the case when the graph is a  clique.
We have also  considered  sequence of Erd\H{o}s-R\'enyi random graphs and established sufficient criteria in terms of the growth rates of the average degree that ensure the annealed and quenched  limit of the occupancy process on any finite time interval to coincide with that in the clique.

The long time behavior of the occupancy measure process associated with
 the above graph sequences is an important and challenging  open question.
Long time properties of the JSQ($d$) scheme have been well studied in the case of a clique.
For example, in~\cite{Mitzenmacher1996, Mitzenmacher01} it is shown that $\pi^N$, the stationary measure of the occupancy process of the $N$-th system, converges in distribution to $\delta_{\qq^*}$, where $\qq^*$ is the unique fixed point of the limiting deterministic dynamical system $\qq(\cdot)$.
Roughly speaking such a result says that the limits  $t\to\infty$ and $N\to\infty$ can be interchanged.
Based on Theorems \ref{th:deterministic}--\ref{thm:npn_rate_quench}, it is natural to conjecture that a similar interchangeability also holds for more general graphs considered in this work.
However, the setting here is significantly harder, in particular, the occupancy process is not any more a Markov process.
One may conjecture that with $\pi^N$ replaced by the time asymptotic limit of the law of occupancy process, the convergence $\pi^N \to \delta_{\qq^*}$ still holds. However,
currently even the existence of such a time asymptotic limit is not clear.

\section*{Acknowledgement}
Research of AB has been partially supported by the National Science Foundation (DMS-1305120), the Army Research Office (W911NF-14-1-0331) and DARPA (W911NF-15-2-0122).
DM was supported by The Netherlands Organization for Scientific Research (NWO) through Gravitation Networks grant 024.002.003, and TOP-GO grant 613.001.012.
The work was initiated during DM's visit to UNC, Chapel hill. 
DM sincerely thanks the hospitality of UNC, Chapel hill for that.
DM thanks Martin Zubeldia for several helpful comments.


\begin{thebibliography}{}

\bibitem[\protect\astroncite{Aghajani and Ramanan}{2017}]{AR17}
Aghajani, R. and Ramanan, K. (2017).
\newblock {The hydrodynamic limit of a randomized load balancing network}.

\bibitem[\protect\astroncite{Azar et~al.}{1994}]{ABKU94}
Azar, Y., Broder, A.~Z., Karlin, A.~R., and Upfal, E. (1994).
\newblock {Balanced allocations}.
\newblock In {\em Proc. STOC '94}, pages 593--602.

\bibitem[\protect\astroncite{Bhamidi et~al.}{2016}]{BhamidiBudhirajaWu2016}
Bhamidi, S., Budhiraja, A., and Wu, R. (2016).
\newblock {Weakly interacting particle systems on inhomogeneous random graphs}.
\newblock {\em arXiv:1612.00801}.

\bibitem[\protect\astroncite{van~der Boor et~al.}{2018}]{BBLM18}
van~der Boor, M., Borst, S.~C., van Leeuwaarden, J. S.~H., and Mukherjee, D.
  (2018).
\newblock {Scalable load balancing in networked systems: Universality
  properties and stochastic coupling methods}.
\newblock In {\em Proc. ICM '18}.

\bibitem[\protect\astroncite{Bramson et~al.}{2012}]{BLP12}
Bramson, M., Lu, Y., and Prabhakar, B. (2012).
\newblock {Asymptotic independence of queues under randomized load balancing}.
\newblock {\em Queueing Syst.}, 71(3):247--292.

\bibitem[\protect\astroncite{Bramson et~al.}{2013}]{BLP13}
Bramson, M., Lu, Y., and Prabhakar, B. (2013).
\newblock {Decay of tails at equilibrium for FIFO join the shortest queue
  networks}.
\newblock {\em Ann. Appl. Probab.}, 23(5).

\bibitem[\protect\astroncite{Budhiraja and Friedlander}{2017}]{BudFri2}
Budhiraja, A. and Friedlander, E. (2017).
\newblock Diffusion approximations for load balancing mechanisms in cloud
  storage systems.
\newblock {\em arXiv:1706.09914}.

\bibitem[\protect\astroncite{Chung and Lu}{2006}]{CL06}
Chung, F.~R. and Lu, L. (2006).
\newblock {\em Complex graphs and networks}.
\newblock Number 107. American Mathematical Soc.

\bibitem[\protect\astroncite{Delattre et~al.}{2016}]{Delattre2016}
Delattre, S., Giacomin, G., and Lu{\c{c}}on, E. (2016).
\newblock A note on dynamical models on random graphs and {F}okker--{P}lanck
  equations.
\newblock {\em J. Stat. Phys.}, 165(4):785--798.

\bibitem[\protect\astroncite{Eschenfeldt and Gamarnik}{2016}]{EG16}
Eschenfeldt, P. and Gamarnik, D. (2016).
\newblock {Supermarket queueing system in the heavy traffic regime. Short queue
  dynamics}.
\newblock {\em arXiv: 1610.03522}.

\bibitem[\protect\astroncite{Fricker and Gast}{2016}]{FG16}
Fricker, C., and Gast, N. (2016). 
newblock {Incentives and redistribution in homogeneous bike-sharing systems with stations of finite capacity.} 
\newblock {\em EURO J. Transp. Logist.}, 5(3), 261--291.

\bibitem[\protect\astroncite{Gast}{2015}]{G15}
Gast, N. (2015).
\newblock {The power of two choices on graphs: the pair-approximation is
  accurate}.
\newblock In {\em MAMA workshop '15}.

\bibitem[\protect\astroncite{Graham}{2005}]{G05}
Graham, C. (2005).
\newblock {Functional central limit theorems for a large network in which
  customers join the shortest of several queues}.
\newblock {\em Probab. Theory Relat. Fields}, 131(1):97--120.

\bibitem[\protect\astroncite{Kenthapadi and Panigrahy}{2006}]{KP06}
Kenthapadi, K. and Panigrahy, R. (2006).
\newblock {Balanced allocation on graphs}.
\newblock In {\em Proc. SODA '06}, pages 434--443.

\bibitem[\protect\astroncite{Kolokoltsov}{2010}]{Kolokoltsov2010}
Kolokoltsov, V.~N. (2010).
\newblock {\em {Nonlinear Markov Processes and Kinetic Equations}}, volume 182
  of {\em Cambridge Tracts in Mathematics}.
\newblock Cambridge University Press.

\bibitem[\protect\astroncite{Kurtz and Xiong}{1999}]{KurtzXiong1999}
Kurtz, T.~G. and Xiong, J. (1999).
\newblock {Particle representations for a class of nonlinear SPDEs}.
\newblock {\em Stoch. Process. Appl.}, 83(1):103--126.

\bibitem[\protect\astroncite{Luczak and McDiarmid}{2006}]{LM06}
Luczak, M.~J. and McDiarmid, C. (2006).
\newblock {On the maximum queue length in the supermarket model}.
\newblock {\em Ann. Probab.}, 34(2):493--527.

\bibitem[\protect\astroncite{Luczak and Norris}{2005}]{LN05}
Luczak, M.~J. and Norris, J. (2005).
\newblock {Strong approximation for the supermarket model}.
\newblock {\em Ann. Appl. Probab.}, 15(3):2038--2061.

\bibitem[\protect\astroncite{Mitzenmacher}{1996}]{Mitzenmacher1996}
Mitzenmacher, M. (1996).
\newblock {\em {The power of two choices in randomized load balancing}}.
\newblock PhD thesis, University of California, Berkeley.

\bibitem[\protect\astroncite{Mitzenmacher}{2001}]{Mitzenmacher01}
Mitzenmacher, M. (2001).
\newblock {The power of two choices in randomized load balancing}.
\newblock {\em IEEE Trans. Parallel Distrib. Syst.}, 12(10):1094--1104.

\bibitem[\protect\astroncite{Mitzenmacher et~al.}{2002}]{MPS02}
Mitzenmacher, M., Prabhakar, B., and Shah, D. (2002).
\newblock {Load balancing with memory}.
\newblock In {\em Proc. FOCS '02}, pages 799--808.

\bibitem[\protect\astroncite{Mukherjee et~al.}{2018}]{MBL17}
Mukherjee, D., Borst, S.~C., and van Leeuwaarden, J. S.~H. (2018).
\newblock {Asymptotically optimal load balancing topologies}.
\newblock {\em Proc. ACM Meas. Anal. Comput. Syst.}, 2(1):1--29.

\bibitem[\protect\astroncite{Mukherjee et~al.}{2018}]{MBLW16-3}
Mukherjee, D., Borst, S.~C., van Leeuwaarden, J. S.~H., and Whiting, P.~A.
  (2018).
\newblock {Universality of power-of-d load balancing in many-server systems}.
\newblock {\em Stoch. Syst.} (to appear).

\bibitem[\protect\astroncite{Peres et~al.}{2015}]{PTW15}
Peres, Y., Talwar, K., and Wieder, U. (2015).
\newblock {Graphical balanced allocations and the (1 + $\beta$)-choice
  process}.
\newblock {\em Random Struc. Algor.}, 47(4):760--775.

\bibitem[\protect\astroncite{Sznitman}{1991}]{Sznitman1991}
Sznitman, A.-S. (1991).
\newblock {\em {Topics in propagation of chaos}}, volume 1464 of {\em Lecture
  Notes in Mathematics}, pages 165--251.
\newblock Springer Berlin Heidelberg.

\bibitem[\protect\astroncite{Tsitsiklis and Xu}{2013}]{TX15}
Tsitsiklis, J. N., and Xu, K. (2017). 
\newblock {Flexible queueing architectures}. 
\newblock {\em Oper. Res.}, 65(5):1398-1413.

\bibitem[\protect\astroncite{Tsitsiklis and Xu}{2011}]{TX11}
Tsitsiklis, J.~N. and Xu, K. (2012).
\newblock {On the power of (even a little) resource pooling}.
\newblock {\em Stoch. Syst.}, 2(1):1--66.

\bibitem[\protect\astroncite{Turner}{1998}]{T98}
Turner, S.~R. (1998).
\newblock {The effect of increasing routing choice on resource pooling}.
\newblock {\em Probab. Eng. Inf. Sci.}, 12(01):109--124.

\bibitem[\protect\astroncite{Vvedenskaya et~al.}{1996}]{VDK96}
Vvedenskaya, N.~D., Dobrushin, R.~L., and Karpelevich, F.~I. (1996).
\newblock {Queueing system with selection of the shortest of two queues: An
  asymptotic approach}.
\newblock {\em Problemy Peredachi Informatsii}, 32(1):20--34.

\bibitem[\protect\astroncite{Wang et~al.}{2016}]{WZYTZ16}
Wang, W., Zhu, K., Ying, L., Tan, J., and Zhang, L. (2016).
\newblock {MapTask Scheduling in MapReduce with Data Locality: Throughput and
  Heavy-Traffic Optimality}.
\newblock {\em IEEE/ACM Trans. Netw.}, 24(1):190--203.

\bibitem[\protect\astroncite{Wieder}{2017}]{W17}
Wieder, U. (2017).
\newblock {Hashing, load balancing and multiple choice}.
\newblock {\em Found. Trends Theor. Comput. Sci.}, 12(3–4):275--379.

\bibitem[\protect\astroncite{Xie et~al.}{2016}]{XYL16}
Xie, Q., Yekkehkhany, A., and Lu, Y. (2016).
\newblock {Scheduling with multi-level data locality: Throughput and
  heavy-traffic optimality}.
\newblock In {\em Proc. INFOCOM '16}, pages 1--9.

\bibitem[\protect\astroncite{Ying}{2017}]{Ying17}
Ying, L. (2017).
\newblock {Stein's method for mean field approximations in light and heavy
  traffic regimes}.
\newblock {\em Proc. ACM Meas. Anal. Comput. Syst.}, 1(1):12.

\end{thebibliography}

\end{document}